\documentclass[11pt]{amsart}
\usepackage[mathscr]{euscript}
\DeclareFontFamily{OT1}{rsfs}{}
\DeclareFontShape{OT1}{rsfs}{n}{it}{<-> rsfs10}{}
\DeclareMathAlphabet{\curly}{OT1}{rsfs}{n}{it}

\makeatletter
\newcommand{\eqnum}{\refstepcounter{equation}\textup{\tagform@{\theequation}}}
\makeatother

\newcommand\beq[1]{\begin{equation}\label{#1}}
\newcommand\eeq{\end{equation}}
\newcommand\beqa{\begin{eqnarray*}}
\newcommand\eeqa{\end{eqnarray*}}

\title[Derived categories of Thaddeus pair moduli spaces]
{Derived categories of Thaddeus pair moduli spaces via d-critical flips}
\date{}
\author{Naoki Koseki}
\address{Graduate School of Mathematical Sciences, The University of Tokyo
3-8-1 Komaba, Meguro, Tokyo, 153-8914 Japan}
\email{koseki@ms.u-tokyo.ac.jp}

\author{Yukinobu Toda}
\address{Kavli Institute for the Physics and Mathematics of the Universe (WPI),The University of Tokyo Institutes for Advanced Study, The University of Tokyo, Kashiwa, Chiba 277-8583, Japan}
\email{yukinobu.toda@ipmu.jp}

\usepackage{amscd}
\usepackage{amsmath}
\usepackage{amssymb}
\usepackage{amsthm}
\usepackage{float}
\usepackage[dvips]{graphicx}
\usepackage{xypic}

\usepackage{array}
\usepackage{amscd}
\usepackage[all]{xy}

\DeclareFontFamily{U}{rsfs}{%
\skewchar\font127}
\DeclareFontShape{U}{rsfs}{m}{n}{%
<-6>rsfs5<6-8.5>rsfs7<8.5->rsfs10}{}
\DeclareSymbolFont{rsfs}{U}{rsfs}{m}{n}
\DeclareSymbolFontAlphabet
{\mathrsfs}{rsfs}
\DeclareRobustCommand*\rsfs{%
\@fontswitch\relax\mathrsfs}

\theoremstyle{plain}
\newtheorem{thm}{Theorem}[section]
\newtheorem{prop}[thm]{Proposition}
\newtheorem{lem}[thm]{Lemma}

\newtheorem{defi}[thm]{Definition}
\newtheorem{rmk}[thm]{Remark}
\newtheorem{cor}[thm]{Corollary}

\newtheorem{prop-defi}[thm]{Proposition-Definition}
\newtheorem{thm-defi}[thm]{Theorem-Definition}
\newtheorem{lem-defi}[thm]{Lemma-Definition}

\newtheorem{exam}[thm]{Example}

\newcommand{\sslash}{/\!\!/}

\newcommand{\cC}{\mathcal{C}}

\newcommand{\eE}{\mathcal{E}}
\newcommand{\fF}{\mathcal{F}}

\newcommand{\hH}{\mathcal{H}}

\newcommand{\mM}{\mathcal{M}}

\newcommand{\oO}{\mathcal{O}}
\newcommand{\pP}{\mathcal{P}}
\newcommand{\qQ}{\mathcal{Q}}

\newcommand{\sS}{\mathcal{S}}

\newcommand{\uU}{\mathcal{U}}
\newcommand{\vV}{\mathcal{V}}
\newcommand{\wW}{\mathcal{W}}
\newcommand{\xX}{\mathcal{X}}
\newcommand{\yY}{\mathcal{Y}}
\newcommand{\zZ}{\mathcal{Z}}

\newcommand{\Hom}{\mathop{\rm Hom}\nolimits}

\newcommand{\dR}{\mathbf{R}}
\newcommand{\dL}{\mathbf{L}}

\newcommand{\Pic}{\mathop{\rm Pic}\nolimits}

\newcommand{\id}{\textrm{id}}

\newcommand{\rk}{\mathop{\rm rk}\nolimits}

\newcommand{\Ext}{\mathop{\rm Ext}\nolimits}
\newcommand{\Spec}{\mathop{\rm Spec}\nolimits}
\newcommand{\rank}{\mathop{\rm rank}\nolimits}
\newcommand{\Coh}{\mathop{\rm Coh}\nolimits}

\newcommand{\cneq}{\mathrel{\raise.095ex\hbox{:}\mkern-4.2mu=}}
\newcommand{\eqcn}{\mathrel{=\mkern-4.5mu\raise.095ex\hbox{:}}}

\newcommand{\Cok}{\mathop{\rm Cok}\nolimits}
\newcommand{\ext}{\mathop{\rm ext}\nolimits}

\newcommand{\Aut}{\mathop{\rm Aut}\nolimits}

\newcommand{\SL}{\mathop{\rm SL}\nolimits}

\newcommand{\Tot}{\mathop{\rm Tot}\nolimits}

 \newcommand{\RHom}{\mathop{\dR\mathrm{Hom}}\nolimits}
\newcommand{\Ker}{\mathop{\rm Ker}\nolimits}

\newcommand{\GL}{\mathop{\rm GL}\nolimits}

\usepackage{mathtools}
\DeclarePairedDelimiter\ceil{\lceil}{\rceil}
\DeclarePairedDelimiter\floor{\lfloor}{\rfloor}

\makeatletter
 \renewcommand{\theequation}{%
   \thesection.\arabic{equation}}
  \@addtoreset{equation}{section}
\makeatother

\begin{document}

\begin{abstract}
We show that the 
moduli spaces of Thaddeus pairs on smooth 
projective curves 
and those of dual pairs 
are related by d-critical flips, which are virtual 
birational transformations introduced by the 
second author. We then prove the 
existence of fully-faithful functors between derived 
categories of coherent sheaves on these moduli 
spaces. 
Our result gives an evidence of a d-critical analogue of 
Bondal-Orlov, Kawamata's D/K equivalence conjecture, and also 
a categorification of wall-crossing formula of Donaldson-Thomas 
type invariants on ADHM sheaves introduced by Diaconescu. 
\end{abstract}

\maketitle

\setcounter{tocdepth}{1}
\tableofcontents

\section{Introduction}
The purpose of this paper is to show the existence of 
fully-faithful functors between derived 
categories of coherent sheaves on moduli spaces of
Thaddeus pairs on smooth projective curves
and those of dual Thaddeus pairs. 
They are smooth projective varieties,
and higher rank generalizations of symmetric products
of curves. In this introduction, we first state the main result of this paper, 
then discuss the motivation of our result
from the view point of d-critical birational geometry, 
and finally give an outline of the proof of the main theorem.

\subsection{Main result}
Let $C$ be a smooth projective curve over $\mathbb{C}$
with genus $g$. 
By definition, a \textit{Thadeus pair} 
on $C$ is a pair~\cite{MR1273268} 
\begin{align}\label{Tpair}
(E, s), \ s \in H^0(C, E)
\end{align}
 where 
$E$ is a semistable vector bundle on $C$,
 satisfying the following stability 
condition: 
there is no non-zero subbundle $F \subsetneq E$
satisfying $\mu(F)=\mu(E)$ and $s \in H^0(C, F)$, 
where $\mu(F) \cneq \chi(F)/\rank(F)$.

We denote by 
\begin{align*}
M^{\rm{T}}(r, d)
\end{align*}
 the moduli space of 
Thaddeus pairs (\ref{Tpair}), such that 
the bundle $E$ satisfies the numerical 
condition 
\begin{align}\label{rank/chi}
(\rank(E), \chi(E))=(r, d).
\end{align}
The moduli space $M^{\rm{T}}(r, d)$ is a 
smooth projective variety 
with dimension $d+r^2(g-1)$. 
We consider the following diagram
\begin{align}\label{diagram:Tpair}
\xymatrix{
M^{\rm{T}}(r, d) \ar[rd]_-{\pi^+} & & \ar[ld]^-{\pi^-}
 M^{\rm{T}}(r, -d) \\
& M(r, d). &
}
\end{align}
Here 
$M(r, d)$ is the coarse moduli space of 
$S$-equivalence classes of 
semistable bundles on $C$ satisfying the condition (\ref{rank/chi}), 
and 
the morphisms $\pi^{\pm}$ are defined by
\begin{align*}
\pi^+(E, s)=[E], \ 
\pi^-(E', s')=[E'^{\vee} \otimes \omega_C]. 
\end{align*}

\begin{thm}\label{thm:Tpair}\emph{(Theorem~\ref{thm:dflip}, Theorem~\ref{thm:pair:window})}
Suppose that $d \ge 0$.

(i) The diagram (\ref{diagram:Tpair})
is a d-critical flip for $d>0$, 
a d-critical flop for $d=0$. 
 
(ii) We have a fully-faithful functor
\begin{align*}
\Phi_M \colon 
D^b(M^{\rm{T}}(r, -d)) \hookrightarrow 
D^b(M^{\rm{T}}(r, d)).
\end{align*}
Here both sides are derived categories of coherent sheaves, 
and the functor $\Phi_M$ is 
an equivalence for $d=0$. 
\end{thm}

Here we refer to Definition~\ref{defi:A4} for the notion of 
\textit{d-critical flips, flops}. 
It is natural to interpret the moduli space $M^{\rm{T}}(r, -d)$ in terms 
of \textit{dual Thaddeus pairs}, i.e. it is the moduli space of pairs
\begin{align*}
(E'', s''), \ s'' \in \Hom(E'', \omega_C)
\end{align*}
where $E''$ is a semistable bundle on $C$ satisfying the condition (\ref{rank/chi}), 
and a stability condition similar to the pair (\ref{Tpair})
(see~Lemma~\ref{lem:dual}).  

In Theorem~\ref{globalkernel}, we will also show 
that the fully-faithful functor
$\Phi_M$ restricted 
to the stable part $M^{\rm{st}}(r, d) \subset M(r, d)$
is given by a Fourier-Mukai functor whose 
kernel is a line bundle on the fiber product of 
the diagram (\ref{diagram:Tpair})
over $M^{\rm{st}}(r, d)$. 
Therefore for $d=0$, the 
functor $\Phi_M$ gives a non-trivial autequivalence of 
$D^b(M^{\rm{T}}(r, 0))$. 
We also remark that
for $r=1$, we have
\begin{align*}
M^{\rm{T}}(1, d)
=S^{g+d-1}(C)
\end{align*}
and Theorem~\ref{thm:Tpair}
in this case is a corollary of the main result of~\cite{Todsemi}. 

\subsection{Motivation and Background}
In~\cite{Toddbir}, the second 
author introduced the notion of 
\textit{d-critical flips, flops,} 
for diagrams of Joyce's d-critical loci~\cite{JoyceD}, 
as an analogy of usual flips, flops, in birational 
geometry. In general they
 are not birational in the usual sense, and should be 
interpreted as virtual birational maps. 
Such a d-critical birational transformation 
typically occurs when we consider wall-crossing diagram of 
moduli spaces of stable objects on Calabi-Yau 3-folds. 

On the other hand if two smooth varieties are 
related by a flip (flop) in the usual sense, 
then Bondal-Orlov~\cite{B-O2} and 
Kawamata~\cite{Ka1}
conjectured
the existence of a fully-faithful 
functor (equivalence) of their derived categories of 
coherent sheaves. This conjecture is 
called a \textit{D/K conjecture}. 
We expect an analogy of D/K conjecture for d-critical loci, i.e. 
under a d-critical flip (flop),
there may exist categorifications of Donaldson-Thomas theory 
which are related by a fully-faithful functor (equivalence).  

In~\cite{Toddbir}, it turned out that wall-crossing 
diagrams of Pandharipande-Thomas stable pair moduli spaces
on Calabi-Yau 3-folds~\cite{PT}, 
used in showing the rationality of their generating 
series~\cite{Tolim2, Tsurvey}, form a d-critical minimal model program.
Based on this observation, we expected 
existence of fully-faithful functors of 
certain categorifications of Donaldson-Thomas theory 
under the above wall-crossing (see~\cite[Section~1]{Todsemi}). 
If this is true, then it gives a link between 
categorifications of wall-crossing formula of 
Donaldson-Thomas theory and 
a d-critical analogue of D/K equivalence conjecture.  

In the previous paper~\cite{Todsemi}, 
the second author studied the above expectation in the case of 
wall-crossing of stable pair moduli spaces
on Calabi-Yau 3-folds, under the assumption that 
the curve class is irreducible and the relevant moduli 
spaces are non-singular.  
More precisely in~\cite{Todsemi}, 
we proved 
the existence of fully-faithful functors of derived categories 
of coherent sheaves on 
stable pair moduli spaces in a wall-crossing diagram 
similar to (\ref{diagram:Tpair}), under the assumption mentioned 
above. 
In this case 
the wall-crossing diagram is a simple d-critical flip, 
and 
the situation is much easier than the diagram (\ref{diagram:Tpair}). 
However if the curve class is not irreducible, then the wall-crossing diagram 
is no more simple, and we cannot apply the above strategy used in~\cite{Todsemi}.  

The result of Theorem~\ref{thm:Tpair} is motivated 
by extending the result of~\cite{Todsemi} for non-irreducible 
curve classes. 
Indeed we will see in Section~\ref{sec:ADHM}
 that the diagram (\ref{diagram:Tpair}) is 
a $\mathbb{C}^{\ast}$-fixed locus of 
a wall-crossing diagram which appeared in~\cite{Tolim2},
for a non-compact Calabi-Yau 3-fold $X$ of the form 
\begin{align}\label{intro:loccurve}
X=\mathrm{Tot}_C(M_1^{\vee} \oplus M_2^{\vee}), \ 
M_1 \otimes M_2 \cong \omega_C^{\vee}. 
\end{align}
Here $M_1$, $M_2$ are line bundles on $C$. 
The relevant stable objects on the above $X$
 are described in terms of 
ADHM sheaves introduced by Diaconescu~\cite{Dia1, Dia}. 
The rank $r$ of our vector bundle corresponds to the curve 
class $r[C]$ where $[C]$ is the class of the zero section 
of $X \to C$. Thus for $r\ge 2$, 
the result of Theorem~\ref{thm:Tpair} gives a
certain extension of the main result of~\cite{Todsemi} 
for non-irreducible (indeed non-primitive) curve
classes, when $X$ is of the form 
(\ref{intro:loccurve}). 

\subsection{Strategy of the proof of Theorem~\ref{thm:Tpair} (ii)}
In~\cite{Todsemi}, we constructed a fully-faithful functor 
of derived categories of stable pair moduli spaces 
by the following steps: we first constructed a fully-faithful 
functor locally on the base, then described the kernel object
explicitly, and finally used the description of the kernel object
to construct a global fully-faithful functor. In our situation
the diagram (\ref{diagram:Tpair}) is not a simple d-critical flip (flop), 
and it is much harder to employ the above strategy, e.g. the description 
of the kernel object is difficult. 
Instead we 
borrow an idea from~\cite{HalpK3}, 
using window subcategories developed in~\cite{MR3327537, MR3895631},
magic window theorem~\cite{MR3698338, HLKSAM}, 
together 
with analytic local description of the diagram (\ref{diagram:Tpair})
proved in~\cite{MR3811778}. 

Here we give an outline of the proof of Theorem~\ref{thm:Tpair} (ii). 
Let $\mM(r, d)$ be the moduli stack of semistable bundles $E$
on $C$
satisfying the condition (\ref{rank/chi2}), 
and $\eE \to C \times \mM(r, d)$ the universal bundle. 
We write $\dR \mathrm{pr}_{\mM\ast}\eE$
as a two term complex of vector bundles
\begin{align*}
\dR \mathrm{pr}_{\mM\ast}\eE=(\fF_0 \stackrel{\psi}{\to} \fF_1)
\end{align*}
where $\mathrm{pr}_{\mM} \colon C \times \mM(r, d) \to \mM(r, d)$
is the projection. We then define
\begin{align*}
w \colon 
\yY \cneq \fF_0 \times_{\mM(r, d)} \fF_1^{\vee}
\to \mathbb{A}^1
\end{align*}
where $w$ is naturally defined using the map $\psi$. 
We show that there exist open immersions
\begin{align}\label{intro:open}
M^{\rm{T}}(r, \pm d) \subset \wW\cneq \{dw=0\}
\end{align}
which are realized as GIT semistable loci 
with respect to some $\mathbb{Q}$-line bundles $L^{\pm}$
on 
$\yY$ restricted to $\wW$. 

We then use the window theorem~\cite{MR3327537, MR3895631}
for the $\mathbb{C}^{\ast}$-equivariant 
derived factorization category $D_{\mathbb{C}^{\ast}}(\yY, w)$, 
and construct subcategories together with equivalences 
\begin{align*}
\cC_{\yY}^{\pm} \subset D_{\mathbb{C}^{\ast}}(\yY, w), \ 
\cC_{\yY}^{\pm} \stackrel{\sim}{\to} 
D_{\mathbb{C}^{\ast}}(\yY_{L^{\pm}\rm{\mathchar`-ss}}, w). 
\end{align*}
Moreover 
by (\ref{intro:open}), 
a version of 
Kn\"orrer periodicity~\cite{MR3071664, MR2982435, MR3631231}
implies the equivalences
\begin{align*}
D^b(M^{\rm{T}}(r, \pm d)) \stackrel{\sim}{\to}
D_{\mathbb{C}^{\ast}}(\yY_{L^{\pm}\rm{\mathchar`-ss}}, w). 
\end{align*}

We are reduced to showing the inclusion 
$\cC_{\yY}^- \subset \cC_{\yY}^+$, which 
is now a local statement on $M(r, d)$. 
The result of~\cite{MR3811778} shows 
that the diagram (\ref{diagram:Tpair})
 is analytic locally on $M(r, d)$
described as moduli spaces of representations of 
a quiver $Q^{\dag}$ with a super-potential. 
The quiver $Q^{\dag}$ is not necessary symmetric, 
but we can find its subquiver $Q_0^{\dag}$
which is symmetric, thus its representation space
is a symmetric representation of a reductive group. 
Then applying the magic window theorem~\cite{MR3698338, HLKSAM} for 
 quasi-symmetric representations
of reductive groups, 
we conclude the inclusion $\cC_{\yY}^- \subset \cC_{\yY}^+$. 

\subsection{Related works}
This paper is regarded as a sequel of the second 
author's previous paper~\cite{Todsemi}, 
where we proved the existence of a fully-faithful 
functor between (simpler) d-critical flips of stable pair 
moduli spaces. 

In~\cite{NP}, it is proved that under wall-crossing of 
pair moduli spaces $(\oO_C \to E)$
there exist fully-faithful functors of 
derived categories. The moduil spaces of stable pairs 
considered in \textit{loc.~cit.~} are birational under wall-crossing, 
and different from the wall-crossing considered here. 
In terms of $\delta$-stability of ADHM sheaves
(see Remark~\ref{rmk:adhm}), the wall-crossing in~\cite{NP} considers 
walls located in $\delta>0$, while our wall-crossing corresponds 
to the wall at $\delta=0$. 

There exist some other works
proving the existence of fully-faithful functors 
between derived categories of stable objects. 
In~\cite{MR3652079}, Ballard proved the existence of 
full-faithful functor for wall-crossing of stable 
sheaves on some rational surfaces. 
In~\cite{HalpK3}, Halpern-Leistner announces the result that 
the derived categories of stable objects on K3 surfaces
are equivalent under wall-crossing. 
In~\cite{MR3652079, HalpK3}, they use window subcategories 
to show their results. 
The moduli spaces in \textit{loc.~cit.~}
are birational under 
wall-crossing, so the situation is different from ours.
However we are much influenced by their works, 
and borrowed several ideas from them.

 \subsection{Acknowledgements}
We are grateful to Yuki Hirano and
Daniel Halpern-Leistner
for valuable discussions.
N.~K.~is supported by the program 
for Leading Graduate Schools, MEXT, Japan, 
and by Grant-in-Aid for JSPS Research Fellow 17J00664.
Y.~T.~is supported by World Premier International Research Center
Initiative (WPI initiative), MEXT, Japan, and Grant-in Aid for Scientific
Research grant (No. 26287002) from MEXT, Japan.

\subsection{Notation and convention}
In this paper, all the schemes and stacks are defined over $\mathbb{C}$. 
For a scheme or a stack $M$, we always denote by $D^b(M)$ the bounded
derived category of coherent sheaves on $M$.

\section{Derived factorization categories under variation of GIT quotients}
In this section, we recall some necessary background 
on derived factorization categories and their window 
subcategories. We will also show some inclusion of 
window subcategories as an application of magic window 
theorem~\cite{MR3698338, HLKSAM}. 

\subsection{Kempf-Ness stratification}\label{subsec:KN}
Let $G$ be a reductive algebraic group, 
with maximal torus $T \subset G$. 
We always denote by $M$ the character lattice of $T$
and $N$ the cocharacter lattice of $T$, i.e. 
\begin{align*}
M=\Hom_{\mathbb{Z}}(T, \mathbb{C}^{\ast}), \ 
N=\Hom_{\mathbb{Z}}(\mathbb{C}^{\ast}, T). 
\end{align*}
The subspace $M_{\mathbb{R}}^W \subset M_{\mathbb{R}}$ is 
defined to be the 
Weyl-invariant subspace. 

Below we follow the convention of~\cite[Section~2.1]{MR3327537} for 
Kempf-Ness stratification associated with GIT quotients. 
Let $X$ be a smooth variety, projective over an affine 
variety, with a $G$-action. 
For a $G$-linearized ample line bundle $L$ on $X$, 
we have the open subset of $L$-semistable points
\begin{align*}
X_{L\rm{\mathchar`-ss}} \subset X.
\end{align*}
By the Hilbert-Mumford criterion, 
$X_{L\rm{\mathchar`-ss}}$ is characterized 
by the set of points $x \in X$ such that 
for any one parameter subgroup $\lambda \colon \mathbb{C}^{\ast} \to G$
for which 
the limit 
\begin{align*}
y=\lim_{t\to 0}\lambda(t)(x) \in X
\end{align*}
exists, we have 
$\mathrm{weight}_{\lambda}(L|_{y}) \ge 0$. 
Also by fixing a Weyl-invariant 
norm $\lvert \ast \rvert$ on $N_{\mathbb{R}}$, 
we have the associated Kempf-Ness (KN) stratification 
\begin{align}\label{KN:strata}
X=X_{L\rm{\mathchar`-ss}} \sqcup S_{1} \sqcup S_{2} \sqcup \ldots. 
\end{align}
Here for each $\alpha$ there exists a 
one parameter subgroup $\lambda_{\alpha} \colon \mathbb{C}^{\ast} \to T$, a connected component (called center)
$Z_{\alpha}$ of the 
$\lambda_{\alpha}$-fixed part 
of 
$X \setminus \cup_{\alpha'<\alpha} S_{\alpha}$
such that 
\begin{align*}
S_{\alpha}=G \cdot Y_{\alpha}, \ 
Y_{\alpha}\cneq \{ x \in X: 
\lim_{t \to 0}\lambda_{\alpha}(t)(x) \in Z_{\alpha}\}. 
\end{align*}
Moreover by setting 
\begin{align*}
\mu_{\alpha}=-\frac{\mathrm{weight}_{\lambda_{\alpha}}(L|_{Z_{\alpha}})}{
\lvert \lambda_{\alpha} \rvert} \in \mathbb{R}
\end{align*}
we have 
the inequalities
$\mu_1>\mu_2>\cdots>0$.

By taking the quotient stacks of the 
stratification (\ref{KN:strata}), we have the 
stratification of the quotient stack $\xX=[X/G]$
\begin{align}\label{Theta:strata}
\xX=\xX_{L\rm{\mathchar`-ss}} \sqcup \sS_{1} \sqcup \sS_2 \sqcup \ldots.
\end{align} 

\begin{rmk}
The stratification (\ref{Theta:strata}) is the $\Theta$-stratification
introduced in~\cite{Halpinstab}, 
which is intrinsic to the stack $\xX$, uniquely 
determined by $L \in \Pic(\xX)$ and some 
$b \in H^4(BG)$
pulled back via $\xX \to BG$
(see~\cite[Example~4.13]{Halpinstab}). 
\end{rmk}

\subsection{Derived factorization categories}\label{subsec:dfac}
Suppose that there exists a subtorus 
$\mathbb{C}^{\ast} \subset G$ which is contained in 
the center of $G$. 
We 
set $\mathbb{P}(G) \cneq G/\mathbb{C}^{\ast}$ and 
define 
\begin{align}\label{Gtilde}
\widetilde{G} \cneq G \times_{\mathbb{P}(G)} G. 
\end{align}
Here $G \to \mathbb{P}(G)$ is the natural quotient map. 
We have the exact sequence of 
algebraic groups
\begin{align}\label{exact:G}
1 \to G \stackrel{\Delta}{\to} \widetilde{G} \stackrel{\tau}{\to} \mathbb{C}^{\ast} \to 1.
\end{align}
Here $\Delta$ is the diagonal embedding, and $\tau$
sends $(g_1, g_2)$ to $g_1 g_2^{-1}$. Below we regard $G$ as a subgroup of $\widetilde{G}$
by the diagonal $\Delta$, and 
whenever we have a $\widetilde{G}$-action 
we also regard it as a $G$-action by the 
embedding $\Delta$. 

The exact sequence (\ref{exact:G}) splits non-canonically, 
i.e. for any $k \in \mathbb{Z}$ the map 
\begin{align}\label{split:G}
\gamma_k \colon 
\mathbb{C}^{\ast} \to \widetilde{G}, \ 
t \mapsto (t^{k}, t^{k-1})
\end{align}
gives a splitting 
of (\ref{exact:G}), and each choice of $k$ 
gives an isomorphism $G\times \mathbb{C}^{\ast} \stackrel{\cong}{\to} \widetilde{G}$. 

Suppose that the $G$-action on $X$ extends 
to a $\widetilde{G}$-action on it. 
By choosing a splitting (\ref{split:G}), 
this is equivalent to giving an auxiliary
$\mathbb{C}^{\ast}$-action on $X$ which commutes with the 
$G$-action.
In particular, such a $\widetilde{G}$-action must preserve 
 the KN stratification (\ref{KN:strata}). 
Let $w \in \Gamma(X, \oO_X)$ be $\tau$-semi invariant, i.e. 
$g^{\ast}w=\tau(g) w$ for any $g \in \widetilde{G}$. 
Equivalently $w$ is $G$-invariant, 
so is a map of stacks 
\begin{align*}
w \colon \xX=[X/G] \to \mathbb{A}^1
\end{align*}
and $\mathbb{C}^{\ast}$-weight one for any 
choice of splitting (\ref{split:G}). 
Given data as above, the \textit{derived factorization category}
 \begin{align}\label{der:fact}
 D_{\mathbb{C}^{\ast}}(\xX, w)
 \end{align}
 is defined to be the triangulated category, 
whose objects consist of $\widetilde{G}$
\textit{-equivariant factorizations of }$w$, i.e. 
sequences of $\widetilde{G}$-equivariant morphisms of 
$\widetilde{G}$-equivariant coherent sheaves $\pP_0$, $\pP_1$ on $X$
\begin{align}\label{factorization}
\pP_0 \stackrel{f}{\to} \pP_1 \stackrel{g}{\to} \pP_0 
\langle 1 \rangle
\end{align}
satisfying the following: 
\begin{align*}
f \circ g=\cdot w, \ g \circ f=\cdot w.
\end{align*}
Here $\langle n \rangle$ 
means the twist by the $\widetilde{G}$-character
$\tau^n$. 
The category (\ref{der:fact}) is defined to be
 the localization 
of the homotopy category of 
the factorizations (\ref{factorization})
by its subcategory of acyclic factorizations. 
For details, see~\cite{MR3366002}. 

\subsection{{K}n\"orrer periodicity}
For the later use, we recall a version of 
Kn\"orrer periodicity of derived factorization 
categories proved in~\cite{MR3071664, MR2982435, MR3631231}
in a general setting. 
Let $Y$ be a smooth variety with a $G$-action, 
and $\fF \to Y$ be an algebraic $G$-equivariant 
vector bundle. Let 
$\psi \colon Y \to \fF$ be a $G$-equivariant 
regular section of it, i.e. its zero locus 
\begin{align*}
Z\cneq (\psi=0) \subset Y
\end{align*}
has codimension equals to the rank of $\fF$. 
The section $\psi$ naturally defines the morphism
\begin{align*}
w_{\psi} \colon 
\fF^{\vee} \to \mathbb{A}^1
\end{align*}
by sending $(y, v)$ for $y \in Y$ and $v \in \fF|_{y}^{\vee}$ 
to $\langle \psi(y), v \rangle$. 
We have the following diagram
\begin{align*}
\xymatrix{
\fF|_{Z}^{\vee} \ar@<-0.3ex>@{^{(}->}[r]^-{i} \ar[d]_-{p} & \fF^{\vee} \ar[r]^-{w_{\psi}} \ar[d] & \mathbb{A}^1 \\
Z \ar@<-0.3ex>@{^{(}->}[r] & Y. &
}
\end{align*}
We have the following lemma 
which is obvious from the definition of $w_{\psi}$: 
\begin{lem}\label{lem:Z}
For a regular section $\psi$, suppose that $Z$ is 
non-singular. 
Then we have $\{dw_{\psi}=0\} =Z$. 
Here $Z$ is embedded into $\fF^{\vee}$ by the zero 
section $Y \to \fF^{\vee}$. 
\end{lem}
The following is the version of {K}n\"orrer periodicity
we use:  
\begin{thm}\emph{(\cite{MR3071664, MR2982435, MR3631231})}\label{thm:knoer}
Suppose that the $G$-action on $\fF^{\vee}$ 
extends to a $\widetilde{G}$-action 
such that, for some splitting $\gamma_k$
in (\ref{split:G}),
the $\mathbb{C}^{\ast}$ acts on fibers of $\fF^{\vee} \to Y$ with 
weight one. 
Then 
we have the equivalences
\begin{align*}
\xymatrix{
 D^b([Z/G])  \ar[r]^-{\sim}_-{i_{\ast} \circ p^{\ast}} &
 D^b([\fF^{\vee}/(G \times \mathbb{C}^{\ast})], w_{\psi}) \ar[r]^-{\sim}_-{\gamma_k} & 
 D_{\mathbb{C}^{\ast}}([\fF^{\vee}/G], w_{\psi}).
 }
\end{align*}
\end{thm}

\subsection{Window subcategories}\label{subsec:window}
We return to the situation in Section~\ref{subsec:KN}, \ref{subsec:dfac}. 
For a KN-stratification (\ref{KN:strata}), 
let $\eta_{\alpha} \in \mathbb{Z}_{\ge 0}$ be defined by 
\begin{align}\label{eta:alpha}
\eta_{\alpha}=\mathrm{weight}_{\lambda_{\alpha}}(\det(N_{S_{\alpha}/X}^{\vee})).\end{align}
\begin{defi}\label{defi:window}
For $\delta \in \Pic_G(X)_{\mathbb{R}}$, 
the \textit{window subcategory}
\begin{align*}
\cC_{\delta} \subset D_{\mathbb{C}^{\ast}}(\xX, w)
\end{align*}
is defined to be the 
triangulated subcategory consisting of factorizations (\ref{factorization})
such that each derived restriction
$\pP_{\bullet}|_{Z_{\alpha}}$ is isomorphic to 
a factorization 
$\qQ_{\bullet, \alpha}$ of $w|_{Z_{\alpha}}$
satisfying the condition
\begin{align*}
\mathrm{weight}_{\lambda_{\alpha}}(\qQ_{\bullet, \alpha})
\subset \mathrm{weight}_{\lambda_{\alpha}}(\delta|_{Z_{\alpha}})+
\left[-\frac{\eta_{\alpha}}{2}, 
\frac{\eta_{\alpha}}{2} \right). 
\end{align*}
Here we note that $\mathrm{weight}_{\lambda_{\alpha}}(\qQ_{\bullet, \alpha})$
is a set of integers. 
\end{defi}
The following is a version of window theorem 
for derived categories of GIT quotients. 
\begin{thm}\emph{(\cite{MR3327537, MR3895631})}\label{thm:window}
The composition
\begin{align*}
\cC_{\delta} \subset D_{\mathbb{C}^{\ast}}(\xX, w)
\stackrel{\rm{res}}{\to} D_{\mathbb{C}^{\ast}}(\xX_{\rm{ss}}, w|_{X_{\rm{ss}}})
\end{align*}
is an equivalence of triangulated categories. 
Here the right arrow is the restriction functor 
to the open substack $\xX_{\rm{ss}} \subset \xX$.
\end{thm}
\begin{proof}
A version of Theorem~\ref{thm:window}
without an auxiliary $\mathbb{C}^{\ast}$-action is stated 
in~\cite[Corollary~3.2.2, Proposition~3.3.2]{MR3895631}
together with~\cite[Remark~3.2.10]{MR3895631}
(also see~\cite[Proposition~5.5, Example~5.7]{MR3327537} for a version
of singularity categories). 
The same arguments apply in the presence of 
$\mathbb{C}^{\ast}$-action without any modification. 
\end{proof}

\subsection{Variation of GIT quotients for linear representations}\label{subsec:VGIT}
In the above situation, suppose 
furthermore that 
$X$ is a 
linear representation of $\widetilde{G}$
which decomposes into
\begin{align*}
X=X_0 \oplus X_1
\end{align*}
as $\widetilde{G}$-representations such that 
$X_0$ is a quasi-symmetric $G$-representation.
Here a $G$-representation $X_0$ is called \textit{quasi-symmetric} 
if $\beta_i \in M$ for $1\le i\le \dim X_0$
are the $T$-weights of $X_0^{\vee}$, then 
for any line 
$l \subset M_{\mathbb{R}}$ we have 
\begin{align*}
\sum_{\beta_i \in l} \beta_i=0.
\end{align*}
In particular, a symmetric representation is quasi-symmetric. 
Let $\overline{\Sigma} \subset M_{\mathbb{R}}$ be 
the convex hull of the $T$-characters of 
$\bigwedge^{\ast}(X_0^{\vee})$. 
Then an element $\chi \in M_{\mathbb{R}}$ is called 
\textit{generic} if it is contained in the linear span of 
$\overline{\Sigma}$ but is not parallel to any face of $\overline{\Sigma}$. 

Let $\chi_0 \colon G \to \mathbb{C}^{\ast}$
be a $G$-character. 
We consider KN stratifications of $X$
and $\xX=[X/G]$
with respect to 
$\oO(\chi_0^{\pm 1})$
\begin{align}\label{KN:X}
X=X_{\rm{ss}}^{\pm} \sqcup S_1^{\pm} \sqcup S_2^{\pm} \sqcup \cdots, \quad 
\xX=\xX_{\rm{ss}}^{\pm} \sqcup \sS_1^{\pm} \sqcup \sS_2^{\pm} \sqcup \cdots 
\end{align}
and take the associated window 
subcategories for $\delta \in \Pic_G(X)_{\mathbb{R}}=M_{\mathbb{R}}^W$
\begin{align*}
\cC_{\delta}^{\pm} \subset D_{\mathbb{C}^{\ast}}(\xX, w).
\end{align*}
We denote by 
$(X_0)_{\rm{ss}}^{\pm} \subset X_0$
the open subsets of $\oO(\chi_0^{\pm 1})$-semistable 
points in $X_0$. 
Note that by the definition of GIT stability, we have 
\begin{align*}
(X_0)_{\rm{ss}}^{\pm} \times X_1 \subset X_{\rm{ss}}^{\pm}. 
\end{align*}
We have the following: 
\begin{prop}\label{prop:magic}
Suppose that 
$\chi_0 \in M_{\mathbb{R}}^W$ is generic
and 
the following condition holds:
\begin{align}\label{ss:minus}
X_{\rm{ss}}^{-} = (X_0)_{\rm{ss}}^{-} \times X_1.
\end{align}
Then for 
any $\delta_0 \in M_{\mathbb{R}}^W$
and $0<\varepsilon \ll 1$, 
by setting $\delta_1=\delta_0+\varepsilon \cdot \chi_0$
we have 
$\cC_{\delta_1}^{-} \subset 
\cC_{\delta_1}^+$. 
\end{prop}
\begin{proof}
The proposition is proved by applying the argument of 
Magic window theorem for quasi-symmetric 
representations~\cite{HLKSAM}, 
proved using combinatorial arguments
in~\cite{MR3698338}. 
Following~\cite[Section~2.2]{HLKSAM}, 
we define 
\begin{align*}
\mathbb{L} \cneq [X_0^{\vee}]-[\mathfrak{g^{\vee}}] 
\in K_0(\mathrm{Rep}(T))=\mathbb{Z}[M].
\end{align*} 
For any one 
parameter subgroup $\lambda \colon \mathbb{C}^{\ast} \to T$, 
we define $\mathbb{L}^{\lambda>0}$
to be the projection of this class onto the subspace
spanned by weights which pair positively with $\lambda$. 
We define
\begin{align*}
\overline{\eta}_{\lambda} \cneq 
\langle \lambda, \mathbb{L}^{\lambda>0} \rangle
\in \mathbb{Z}. 
\end{align*}
Then we define 
$\overline{\nabla} \subset M_{\mathbb{R}}$ by 
\begin{align}\label{defi:nabla}
\overline{\nabla}\cneq \left\{ \chi \in M_{\mathbb{R}} : 
\langle \lambda, \chi \rangle \in
\left[-\frac{\overline{\eta}_{\lambda}}{2}, 
\frac{\overline{\eta}_{\lambda}}{2}  \right] 
\mbox{ for all } \lambda \colon \mathbb{C}^{\ast} \to T  
\right\}. 
\end{align}
For a region $\Omega \subset M_{\mathbb{R}}$, 
we denote by 
\begin{align}\label{def:MOmega}
\mM(\Omega) \subset
D_{\mathbb{C}^{\ast}}(\xX, w)
\end{align}
the triangulated subcategory 
consisting of
factorizations (\ref{factorization})
such that each $\pP_i$ is
isomorphic to $V \otimes_{\mathbb{C}} \oO_X$
for a $\widetilde{G}$-representation $V$
whose $T$-weights are contained in $\Omega$. 
We first claim the inclusion 
\begin{align}\label{M:C}
\mM(\delta_1+\overline{\nabla}) \subset \cC_{\delta_1}^{\pm}
\end{align}
for $\delta_1=\delta_0+\varepsilon \cdot \chi_0$
with $0<\varepsilon \ll 1$. 

Let $\lambda_{\alpha}^{\pm} \colon \mathbb{C}^{\ast} \to T$
be one parameter subgroups for the KN stratifications (\ref{KN:X}), 
and $Z_{\alpha}^{\pm} \subset S_{\alpha}^{\pm}$ the centers. 
For a $\widetilde{G}$-representation $V$
whose $T$-weights are contained in $\delta_1+\overline{\nabla}$, 
we have 
\begin{align*}
\mathrm{weight}_{\lambda_{\alpha}^{\pm}}(V \otimes_{\mathbb{C}}
\oO_X|_{Z_{\alpha}^{\pm}})
\subset \langle \delta_1, \lambda_{\alpha}^{\pm}\rangle+
\left[-\frac{\overline{\eta}_{\lambda_{\alpha}^{\pm}}}{2}, 
\frac{\overline{\eta}_{\lambda_{\alpha}^{\pm}}}{2}
  \right]
\end{align*} 
by the definition of $\overline{\nabla}$. 
We have the inequality
\begin{align*}
\overline{\eta}_{\lambda_{\alpha}^{\pm}} \le 
\langle \lambda_{\alpha}^{\pm}, 
(X^{\vee})^{\lambda_{\alpha}^{\pm}>0} 
-(\mathfrak{g}^{\vee})^{\lambda_{\alpha}^{\pm}>0}\rangle
=\eta_{\alpha}^{\pm}
\end{align*}
where the latter is defined as in (\ref{eta:alpha})
for the KN stratifications (\ref{KN:X}).  
For the second equality, see~\cite[Equation~(4)]{MR3327537}. 
On the other hand, we have 
$\langle \chi_0^{\pm 1}, \lambda_{\alpha}^{\pm}\rangle <0$
by the property of KN stratifications. 
Therefore by taking $0<\varepsilon \ll 1$ so
that $\langle \delta_1, \lambda_{\alpha}^{\pm} \rangle + \eta_{\alpha}^{\pm}/2$
are not integers, we have 
\begin{align*}
\mathrm{weight}_{\lambda_{\alpha}^{\pm}}(V \otimes_{\mathbb{C}}
\oO_X|_{Z_{\alpha}^{\pm}})
\subset \langle \delta_1, \lambda_{\alpha}^{\pm}\rangle+
\left[-\frac{\eta_{\alpha}^{\pm}}{2}, 
\frac{\eta_{\alpha}^{\pm}}{2} 
  \right). 
\end{align*} 
Therefore the inclusion (\ref{M:C}) holds. 

We consider the following commutative diagram
\begin{align}\label{dia:MW}
\xymatrix{
\cC_{\delta_1}^- \ar@<-0.3ex>@{^{(}->}[r]  
& D_{\mathbb{C}^{\ast}}(\xX, w) 
\ar@{=}[d]
\ar[r]^-{\rm{res}} & D_{\mathbb{C}^{\ast}}(\xX_{\rm{ss}}^{-}, w) 
\ar[d]^-{\cong}_-{\rm{res}} \\
\mM(\delta_1+\overline{\nabla}) 
\ar@<-0.3ex>@{^{(}->}[u] \ar@<-0.3ex>@{^{(}->}[r]
& D_{\mathbb{C}^{\ast}}(\xX, w)
\ar[r]^-{\rm{res}} & 
D_{\mathbb{C}^{\ast}}([((X_{0})_{\rm{ss}}^{-} \times X_1)/G], w).
} 
\end{align}
The top composition is an equivalence by Theorem~\ref{thm:window}
and the right restriction functor is an equivalence by 
the assumption (\ref{ss:minus}).

Below we show that 
 the bottom composition of the diagram (\ref{dia:MW})
is essentially surjective. 
Let 
\begin{align*}
\mM'(\delta_1+\overline{\nabla})
\subset D^b([X_0/\widetilde{G}]), \ 
\mM''(\delta_1+\overline{\nabla})
\subset D^b([X/\widetilde{G}])
\end{align*}
be the triagulated subcategories generated by 
$\widetilde{G}$-equivariant vector bundles 
on $X_0$, $X$ whose $T$-weights are contained in $\delta_1+\overline{\nabla}$. 
By~\cite[Corollary~2.9]{HLKSAM}, any facet of 
$\overline{\nabla}$ is parallel to some facet
of $\overline{\Sigma}$. 
Therefore the genericity condition of $\chi_0$ implies that 
$\partial(\delta_1+\overline{\nabla}) \cap M=\emptyset$, 
and $[(X_0)_{T-{\rm{ss}}}^{-}/T]$ is Deligne-Mumford 
by~\cite[Proposition~2.1]{HLKSAM}. 
Therefore the assumption of~\cite[Proposition~3.11]{HLKSAM}
is satisfied, which shows that 
the following composition is an equivalence
(Magic window theorem in~\cite[Theorem~3.2]{HLKSAM})
\begin{align*}
\mM'(\delta_1+\overline{\nabla})
\hookrightarrow D^b([X_0/\widetilde{G}])
\to D^b([(X_0)_{\rm{ss}}^-/\widetilde{G})]).
\end{align*}
The above result is stated in \textit{loc.~cit.~} without 
an auxiliary $\mathbb{C}^{\ast}$-action, but the 
same argument applies in the presence of an auxiliary $\mathbb{C}^{\ast}$-action, see~\cite[Corollary~5.2]{HLKSAM}. 

Also  
$D^b([((X_0)_{\rm{ss}}^- \times X_1)/\widetilde{G}])$
is generated by images of the pull-backs by 
the projection 
\begin{align*}
[((X_0)_{\rm{ss}}^- \times X_1)/\widetilde{G}] 
\to [(X_0)_{\rm{ss}}^-/\widetilde{G}]
\end{align*}
as it is a total space of a vector bundle on 
$[(X_0)_{\rm{ss}}^-/\widetilde{G}]$. It follows that 
$D^b([((X_0)_{\rm{ss}}^- \times X_1)/\widetilde{G}])$
is generated by $V \otimes_{\mathbb{C}} \oO_{X}$
for $\widetilde{G}$-representation $V$ whose $T$-weights are 
contained in $\delta_1+\overline{\nabla}$. 
In other words, the composition 
\begin{align}\label{compose:M''}
\mM''(\delta_1+\overline{\nabla}) \hookrightarrow D^b([X/\widetilde{G}])
\to D^b([(X_0)_{\rm{ss}}^- \times X_1/\widetilde{G}])
\end{align}
is essentiallly surjective. 
The above composition is also fully-faithful by
a version of the diagram (\ref{dia:MW}) without $w$, 
so the functor (\ref{compose:M''}) is an equivalence. 

By the equivalence (\ref{compose:M''}), we see that 
any $\widetilde{G}$-equivariant factorization 
on $(X_0)_{\rm{ss}}^- \times X_1$
is a restriction of
some object in $\mM(\delta_1+\overline{\nabla})$. 
Indeed for a factorization 
\begin{align}\label{fact:Pss}
\pP_0 \stackrel{f}{\to} \pP_1 \stackrel{g}{\to} \pP_0\langle 1\rangle
\end{align}
on $(X_0)_{\rm{ss}}^- \times X_1$, 
each $\pP_i$ is resolved by a complex of 
vector bundles $\vV_i^{\bullet}$ 
in 
$\mM''(\delta_1+\overline{\nabla})$ restricted to 
$(X_0)_{\rm{ss}}^- \times X_1$. 
By the equivalence (\ref{compose:M''}), 
the morphisms $f$, $g$ can be lifted to morphisms of complexes 
$\vV_0^{\bullet} \to \vV_1^{\bullet} \to \vV_0^{\bullet}\langle 1 \rangle$
on $X$. 
By taking the totalizations, we obtain a factorization in 
$\mM(\delta_1+\overline{\nabla})$ whose restriction 
to $(X_0)_{\rm{ss}}^- \times X_1$ gives (\ref{fact:Pss}). 
Therefore 
the bottom arrow of (\ref{dia:MW}) is essentially surjective. 
The above arguments show that
$
\cC_{\delta_1}^{-}=\mM(\delta_1+\overline{\nabla}) \subset 
\cC_{\delta_1}^+$
as desired. 
%
\end{proof}

We will show that a similar result of Proposition~\ref{prop:magic}
also holds after taking the direct sum with some 
symmetric representations. 
Let $W$ be a finite dimensional $G$-representation
whose weights of $\mathbb{C}^{\ast} \subset G$ are one. 
We set
\begin{align*}
X_W=X \oplus W \oplus W^{\vee}, \ 
\xX_W=[X_W/G]. 
\end{align*}
We have the KN stratification 
for the $G$-action on $X_W$ 
with respect to $\oO(\chi_0^{\pm 1})$
\begin{align}\label{KN:W}
X_W=(X_W)_{\rm{ss}}^{\pm} \sqcup S_{W, 1}^{\pm} \sqcup S_{W, 2}^{\pm}\cdots, \ 
\xX_W=(\xX_W)_{\rm{ss}}^{\pm} \sqcup \sS_{W, 1}^{\pm} \sqcup \sS_{W, 2}^{\pm} \cdots. 
\end{align}
The $G$-action on $X$ is naturally extended to a $\widetilde{G}$-action 
on 
$X_W$ given by 
\begin{align*}
(g, g')(x, v, v')=((g, g')(x), g(v), g'(v'))
\end{align*}
for $x \in X$, $v \in W$ and $v' \in W^{\vee}$. 
Let $w^{\star}$ be given by 
\begin{align*}
w^{\star} \colon X_W \to \mathbb{C}, \ 
w^{\star}(x, v, v')=w(x)+\langle v, v' \rangle. 
\end{align*}
For $\delta \in M_{\mathbb{R}}^W$, 
we have the associated window subcategories 
with respect to the stratifications (\ref{KN:W})
\begin{align*}
\cC_{W, \delta}^{\pm} \subset D_{\mathbb{C}^{\ast}}(\xX_W, w^{\star}). 
\end{align*}
\begin{thm}\label{thm:magic}
Under the same assumption of Proposition~\ref{prop:magic}, 
we have 
$\cC_{W, \delta_2}^- \subset \cC_{W, \delta_2}^+$
for $\delta_2=\varepsilon \cdot \chi_0$ with $0<\varepsilon \ll 1$. 
\end{thm}
\begin{proof}
Note that if $(X_W)_{\rm{ss}}^-=(X_0 \times W \times W^{\vee})_{\rm{ss}}^- \times X_1$, 
then the theorem follows from Proposition~\ref{prop:magic}. 
However this is not the case in general, and 
we will prove the theorem by comparing 
magic windows under 
the {K}n\"orrer periodicity equivalence. 
Below, we use the notation in the proof of Proposition~\ref{prop:magic}.

Let $\mathbb{L}_W$ be defined by 
\begin{align*}
\mathbb{L}_W=[X_0^{\vee}]+[W]+[W^{\vee}]-[\mathfrak{g}^{\vee}]
\in K(\mathrm{Rep}(T)). 
\end{align*}
Then for each one parameter subgroup 
$\lambda \colon \mathbb{C}^{\ast} \to T$, 
we set
\begin{align}\notag
\overline{\eta}_{W, \lambda} &\cneq 
\langle \lambda, \mathbb{L}_W^{\lambda>0} \rangle \\
\label{eqn:etaW}
&=\overline{\eta}_{\lambda}+\langle \lambda, W^{\lambda>0}\rangle
-\langle \lambda, W^{\lambda<0}\rangle. 
\end{align}
We set $\overline{\nabla}_{W}$ to be 
\begin{align}\label{defi:nabla}
\overline{\nabla}_W\cneq \left\{ \chi \in M_{\mathbb{R}} : 
\langle \lambda, \chi \rangle \in
\left[-\frac{\overline{\eta}_{W, \lambda}}{2}, 
\frac{\overline{\eta}_{W, \lambda}}{2}  \right] 
\mbox{ for all } \lambda \colon \mathbb{C}^{\ast} \to T  
\right\}. 
\end{align}
For $\delta \in M_{\mathbb{R}}^W$, 
we have the subcategory
\begin{align*}
\mM(\delta+\overline{\nabla}_W) \subset D_{\mathbb{C}^{\ast}}(\xX_W, w^{\star})
\end{align*}
similarly to (\ref{def:MOmega}). 
By the argument of Proposition~\ref{prop:magic}, 
we always have the inclusion 
$\mM(\delta+\overline{\nabla}_W) \subset
\cC_{W, \delta}^{\pm}$. 
It is enough to show that 
$\mM(\delta_2+\overline{\nabla}_W) =\cC_{W, \delta_2}^-$ for 
$\delta_2=\epsilon \cdot \chi_0$ with $0<\varepsilon \ll 1$. 

We have the following commutative diagram
\begin{align*}
\xymatrix{
X \times W^{\vee} \ar@<-0.3ex>@{^{(}->}[r]^-{i} \ar[d]_-{p}
 & X_W \ar[rd]^-{w^{\star}} \ar[d] & \\
X \times \{0\} \ar@/_20pt/[rr]_-{w}
\ar@<-0.3ex>@{^{(}->}[r] & X \times W & \mathbb{C}
}
\end{align*}
Here each vertical arrow is a projection. 
The above diagram induces the 
{K}n\"orrer periodicity
equivalence (see~\cite[Theorem~4.2]{MR3631231})
\begin{align}\label{period}
\Phi \cneq i_{\ast} \circ p^{\ast} \colon 
D_{\mathbb{C}^{\ast}}(\xX, w) \stackrel{\sim}{\to} D_{\mathbb{C}^{\ast}}(\xX_W, w^{\star}).
\end{align}
The above equivalence is 
nothing but taking the tensor product 
over $\oO_X$ with 
the factorization 
\begin{align*}
\oO_{X \times W^{\vee}} \to 0 \to \oO_{X \times W^{\vee}}
\langle 1 \rangle
\end{align*}
of the function $w' \colon 
(x, v, v') \mapsto \langle v, v' \rangle$
on $X_W$.
The above factorization is 
isomorphic to 
the Koszul factorization 
on $X_W$ of $w'$ of the form (see~\cite[Proposition~3.20]{MR3270588})
\begin{align}\label{Koszul:fact}
\bigwedge^{\mathrm{even}}W \otimes \oO_{X_W} \to 
\bigwedge^{\mathrm{odd}}W \otimes \oO_{X_W} \to 
\bigwedge^{\mathrm{even}}W \otimes \oO_{X_W}\langle 1 \rangle. 
\end{align}
The Koszul factorization (\ref{Koszul:fact})
 has the minimus $T$-weight $\langle \lambda, W^{\lambda<0}\rangle$
and the maximum $T$-weight $\langle \lambda, W^{\lambda>0}\rangle$. 
Therefore for each $\eE \in \mM(\delta_1+\overline{\nabla})$
in $D_{\mathbb{C}^{\ast}}(\xX, w)$, 
where $\delta_1=\delta_0+\varepsilon \cdot \chi_0$, 
a $T$-weight $\chi$ of $\Phi(\eE)$ satisfies that 
\begin{align*}
-\frac{1}{2}\overline{\eta}_{\lambda}+\langle \lambda, W^{\lambda<0}\rangle
\le \langle \chi-\delta_1, \lambda \rangle 
\le \frac{1}{2}\overline{\eta}_{\lambda}+\langle \lambda, W^{\lambda>0}\rangle. 
\end{align*}
By the identity (\ref{eqn:etaW}), 
we have 
\begin{align*}
\frac{1}{2}\overline{\eta}_{\lambda}+\langle \lambda, W^{\lambda>0}\rangle
&=\frac{1}{2}\overline{\eta}_{W, \lambda}+\frac{1}{2}\langle \lambda, W^{\lambda>0}\rangle
+\frac{1}{2}\langle \lambda, W^{\lambda<0}\rangle \\
&=\frac{1}{2}\overline{\eta}_{W, \lambda}+
\frac{1}{2}\mathrm{weight}_{\lambda}(\det W).
\end{align*}
Similarly we have 
\begin{align*}
-\frac{1}{2}\overline{\eta}_{\lambda}+\langle \lambda, W^{\lambda<0}\rangle
&=-\frac{1}{2}\overline{\eta}_{W, \lambda}+
\frac{1}{2}\mathrm{weight}_{\lambda}(\det W).
\end{align*}
We take $\delta_0=-\frac{1}{2}\det(W) \in \Pic(\xX)_{\mathbb{R}}=M_{\mathbb{R}}^W$ . 
Then the above argument implies that 
\begin{align*}
-\frac{1}{2}\overline{\eta}_{W, \lambda} \le \langle \chi-\delta_2, \lambda\rangle
\le \frac{1}{2}\overline{\eta}_{W, \lambda}. 
\end{align*}
This means that the functor $\Phi$ restricts to the 
functor 
\begin{align*}
\Phi \colon \mM(\delta_1+\overline{\nabla})
\to \mM(\delta_2+\overline{\nabla}_W). 
\end{align*}
We have the commutative diagram
\begin{align}\label{dia:MW2}
\xymatrix{
\mM(\delta_1+\overline{\nabla})
\ar[d]_-{\Phi}
\ar@{=}[r] & 
\cC_{\delta_1}^- \ar@<-0.3ex>@{^{(}->}[r]   \ar@/^20pt/[rr]^-{\Psi_1}
& D_{\mathbb{C}^{\ast}}(\xX, w) 
\ar[d]_-{\Phi}
\ar[r]^-{\rm{res}} & D_{\mathbb{C}^{\ast}}(\xX_{\rm{ss}}^{-}, w) 
\ar[d]_-{\Phi'} \\
\mM(\delta_2+\overline{\nabla}_W) \ar@<-0.3ex>@{^{(}->}[r] &
\cC_{W, \delta_2}^- \ar@<-0.3ex>@{^{(}->}[r] \ar@/_20pt/[rr]_-{\Psi_2}
& D_{\mathbb{C}^{\ast}}(\xX_W, w^{\star})
\ar[r]^-{\rm{res}} & 
D_{\mathbb{C}^{\ast}}((\xX_W)_{\rm{ss}}^-, w^{\star}).
} 
\end{align}
Here the top left identity is proved 
in the proof of Proposition~\ref{prop:magic}, and 
$\Psi_1$, $\Psi_2$ are equivalences by Theorem~\ref{thm:window}. 

We see that the equivalence $\Phi$ descends to the equivalence $\Phi'$ in the 
diagram (\ref{dia:MW2}). 
Since we have 
\begin{align*}
\{dw^{\star}=0\}=\{dw=0\} \times \{0\} \times \{0\}
\subset X \times W \times W^{\vee}
\end{align*} 
we have 
\begin{align}\label{crit:w}
\{dw^{\star}=0\} \cap (X_W)_{\rm{ss}}^- \subset X_{\rm{ss}}^- \times W \times W^{\vee}. 
\end{align}
The equivalence $\Phi'$ is given by the composition of equivalences
\begin{align*}
\Phi' \colon D_{\mathbb{C}^{\ast}}(\xX_{\rm{ss}}^-, w) \stackrel{\sim}{\to}
D_{\mathbb{C}^{\ast}}( [X_{\rm{ss}}^- \times W \times W^{\vee}/G], w^{\star})
\stackrel{\sim}{\to} D_{\mathbb{C}^{\ast}}((\xX_W)_{\rm{ss}}^-, w^{\star}).
\end{align*}
Here the first equivalence is the {K}n\"orrer periodicity
similar to (\ref{period}), 
and the 
second equivalence follows from 
(\ref{crit:w}) together with 
 the fact that 
the derived factorization category 
only depends on an open neighborhood of the critical locus 
(for example see~\cite[Lemma~5.5]{HLKSAM}). 
The resulting equivalence $\Phi'$ fits into the 
commutative diagram (\ref{dia:MW2}) by the construction. 

By the diagram (\ref{dia:MW2}), 
we have the inclusion 
$\Phi(\cC_{\delta_1}^-) \subset \cC_{W, \delta_2}^-$. 
Since $\Psi_1$, $\Psi_2$ and $\Phi'$ are equivalences, 
it follows that 
$\Phi(\cC_{\delta_1}^-)=\cC_{W, \delta_2}^-$. 
Again using the diagram (\ref{dia:MW2}), 
we conclude that 
$\cC_{W, \delta_2}^-=\mM(\delta_2+\overline{\nabla}_W)$ as desired. 
\end{proof}

We have the obvious corollary of Theorem~\ref{thm:magic}:
\begin{cor}\label{cor:magic}
Suppose that
$\chi_0 \in M_{\mathbb{R}}^W$ is generic and
the following condition holds:
\begin{align}\notag
X_{\rm{ss}}^{\pm}= (X_0)_{\rm{ss}}^{\pm} \times X_1.
\end{align}
Then for any finite dimensional $G$-representation $W$, 
we have 
$\cC_{W, \delta_2}^{-} = \cC_{W, \delta_2}^+$
for $\delta_2=\varepsilon \cdot \chi_0$ with $0<\varepsilon \ll 1$.
\end{cor}

\section{Derived categories of representations of quivers}
In this section, we study derived factorization categories associated 
with certain quivers with super-potentials, and show 
an inclusion of the window subcategories under wall-crossing. 
The result of this section is a local version of Theorem~\ref{thm:Tpair}. 

\subsection{Representations of quivers}\label{subsec:quiver}
Let $Q$ be a quiver 
\begin{align*}
Q=(V(Q), E(Q), s, t)
\end{align*}
where $V(Q)$ is the set of vertices, 
$E(Q)$ is the set of edges, and 
$s, t \colon E(Q) \to V(Q)$
are maps which 
correspond to sources and targets of 
edges. 
For $e \in E(Q)$ with $s(e)=i$, $t(e)=j$, 
we write $e=(i \to j)$. 
Below we assume that $Q$ is symmetric, i.e. 
\begin{align*}
\sharp(i \to j)=\sharp(j \to i), \ i, j \in V(Q).
\end{align*}
We fix finite dimensional vector spaces 
$V_i$ for each $i \in V(Q)$, and 
set
\begin{align*}
\mathrm{Rep}(Q)=\prod_{(i \to j) \in E(Q)}
\Hom(V_i, V_j),
\end{align*}
The algebraic group 
$G$ defined by
\begin{align*}
G=\prod_{i \in V(Q)} \GL(V_i)
\end{align*}
acts on $\mathrm{Rep}(Q)$ by the conjugation, 
which descends to the action 
of $\mathbb{P}(G) = G/\mathbb{C}^{\ast}$
where $\mathbb{C}^{\ast} \subset G$ is the diagonal torus. 
We have the quotient stack together with its 
good moduli space
\begin{align}\label{map:pQ}
p_Q \colon \mathrm{Rep}(Q) \to 
\left[\mathrm{Rep}(Q)/G  \right] 
\stackrel{\iota_Q}{\to} \mathrm{Rep}(Q)\sslash G.
\end{align}
The stack $\left[\mathrm{Rep}(Q)/G  \right]$
is the moduli stack of $Q$-representations 
with dimension vector $(\dim V_i)_{i \in Q}$, 
and $\mathrm{Rep}(Q)\sslash G$ parametrizes
semisimple $Q$-representations 
with the above dimension vector. 

\subsection{Representations of extended quivers}
For each $i \in V(Q)$,
let us take finite sets
\begin{align}\label{set:E}
\mathbf{E}_{i0} \subset \mathbf{E}_{0i}.
\end{align}
We define the extended quiver $Q^{\dag}$ by setting 
\begin{align*}
V(Q^{\dag})=V(Q) \sqcup \{0\}, \ 
E(Q^{\dag})=E(Q) \sqcup_{i \in V(Q)} \mathbf{E}_{0i} \sqcup_{i\in V(Q)} \mathbf{E}_{i0}. 
\end{align*}
Here 
for $e \in \mathbf{E}_{0i}$ (resp.~$e \in \mathbf{E}_{i0}$), its source and target are $0$, $i$
(resp.~$i$, $0$) respectively.  
We set
\begin{align*}
\mathrm{Rep}(Q^{\dag})=
\mathrm{Rep}(Q) \times \prod_{(0 \to i) \in \mathbf{E}_{0i}} V_i
\times \prod_{(i \to 0) \in \mathbf{E}_{i0}}V_i^{\vee}. 
\end{align*}
The space $\mathrm{Rep}(Q^{\dag})$ parametrizes 
$Q^{\dag}$-representations with dimension 
vector $((\dim V_i)_{i \in V(Q)}, 1)$. 

As defined in (\ref{Gtilde}), let 
 $\widetilde{G} = G \times_{\mathbb{P}(G)} G$. 
There is a natural $\widetilde{G}$-action on 
$\mathrm{Rep}(Q^{\dag})$
given by 
\begin{align*}
(g, g')\cdot (x, u, u')=(g(x), g(u), g'(u'))=(g'(x), g(u), g'(u')).
\end{align*}
Here $(g, g') \in \widetilde{G}$ and 
the second identity follows as 
the $G$-action on $\mathrm{Rep}(Q)$ descends to the 
$\mathbb{P}(G)$-action. 

Let $\chi_0 \colon G \to \mathbb{C}^{\ast}$ be the 
character of $G$ defined by
\begin{align}\label{chi0}
\chi_0((g_i)_{i \in V(Q)}) \mapsto \prod_{i \in Q_0} \det(g_i). 
\end{align}
We have the open substacks of 
semistable locus with respect to 
 $\oO(\chi_0^{\pm 1})$
\begin{align*}
M_{Q^{\dag}}^{\pm} \cneq 
[\mathrm{Rep}(Q^{\dag})_{\rm{ss}}^{\pm}/G]
\subset [\mathrm{Rep}(Q^{\dag})/G]. 
\end{align*}
\begin{prop}\label{prop:flip}
The stacks 
$M_{Q^{\dag}}^{\pm}$ are smooth varieties and, 
if $M_{Q^{\dag}}^{-} \neq \emptyset$, then  
the diagram
\begin{align*}
\xymatrix{
M_{Q^{\dag}}^{+} \ar[rd]_-{\pi_Q^+} & & \ar[ld]^-{\pi_Q^-} M_{Q^{\dag}}^- \\
& \mathrm{Rep}(Q^{\dag})\sslash G
}
\end{align*}
is a flip if $\sharp \mathbf{E}_{i0} <\sharp \mathbf{E}_{0i}$ for all $i \in V(Q)$.
It is a flop if $\sharp \mathbf{E}_{i0} =\sharp \mathbf{E}_{0i}$ for all $i \in V(Q)$.
\end{prop} 
\begin{proof}
By a correspondence of GIT stability and King's $\theta$-stability~\cite{Kin},
we see that a $Q^{\dag}$-representation in $\mathrm{Rep}(Q^{\dag})$
is $\oO(\chi_0^{\pm 1})$-semistable 
if and only if it is
$\theta^{\pm}$-semistable where
\begin{align}\label{theta}
\theta^{\pm}=(\theta_i^{\pm})_{i \in V(Q^{\dag})}, \ 
\theta_i^{\pm}=\mp 1 \ \mbox{ for } i \in V(Q), \ \theta_0^{\pm}
=\pm \sum_{i\in V(Q)}\dim V_i.
\end{align}
Here a $Q^{\dag}$-representation $E$ is $\theta^{\pm}$-semistable 
if for any subrepresentation $F \subset E$, we have 
$\theta^{\pm}(\underline{\dim} F) \ge \theta^{\pm}(\underline{\dim}E)=0$, 
where $\underline{\dim}$ is the dimension vector. 
Then the proposition follows from~\cite[Proposition~7.13]{Toddbir}.
\end{proof}

\subsection{Super-potentials on quiver representations}
Suppose that we are given a $G$-equivariant 
morphism of $G$-equivariant vector bundles
on $\mathrm{Rep}(Q)$
\begin{align}\label{mor:s}
\vartheta \colon \left(\prod_{(0 \to i) \in \mathbf{E}_{0i}}V_i\right)
 \otimes \oO_{\mathrm{Rep}(Q)}
\to \left(\prod_{(i \to 0) \in \mathbf{E}_{i0}} V_i\right)
 \otimes \oO_{\mathrm{Rep}(Q)}. 
\end{align}
Then the above morphism induces the 
function
\begin{align}\label{w:repQ}
w \colon \mathrm{Rep}(Q^{\dag}) \to \mathbb{C}, \ 
w(x, u, u')=\langle \vartheta|_{x}(u), u' \rangle. 
\end{align}
for $x \in \mathrm{Rep}(Q)$, 
$u \in \prod_{(0 \to i) \in \mathbf{E}_{0i}}V_i$, 
$u' \in \prod_{(i \to 0) \in \mathbf{E}_{i0}}V_i^{\vee}$. 
The function $w$ is $G$-invariant, so descends to the 
map
\begin{align*}
w^0 \colon \mathrm{Rep}(Q^{\dag})\sslash G \to \mathbb{C}.
\end{align*}
We define $w^{\pm}$ by the commutative diagram
\begin{align}\label{dia:pm}
\xymatrix{
M_{Q^{\dag}}^{+} \ar[r]^-{\pi_Q^+} \ar[dr]_-{w^+} 
& \mathrm{Rep}(Q^{\dag})\sslash G 
\ar[d]^-{w^0}& 
\ar[l]_-{\pi_Q^-} \ar[ld]^-{w^-} M_{Q^{\dag}}^- \\
& \mathbb{C}. &
}
\end{align}
We define the following d-critical loci
\begin{align*}
M_{(Q^{\dag}, w)}^{\pm} \cneq \{dw^{\pm}=0\} \subset M_{Q^{\dag}}^{\pm}. 
\end{align*}
\begin{prop}\label{prop:dcrit}
Suppose that the followings hold:
\begin{align}\label{M:inclu}
&M_{(Q^{\dag}, w)}^+ \subset \left[\left(\mathrm{Rep}(Q) \times 
\prod_{(0 \to i) \in \mathbf{E}_{0i}}V_i \times \{0\}\right)/G  \right], \\ 
\notag
 & M_{(Q^{\dag}, w)}^- \subset \left[\left(\mathrm{Rep}(Q) \times \{0\} \times 
\prod_{(i \to 0) \in \mathbf{E}_{i0}}V_i^{\vee}\right) /G\right].
\end{align}
Then we have the diagram
\begin{align*}
\xymatrix{
M_{(Q^{\dag}, w)}^{+} 
\ar[rd]_-{\pi_Q^+} & & \ar[ld]^-{\pi_Q^-} M_{(Q^{\dag}, w)}^- \\
& \mathrm{Rep}(Q)\sslash G
}
\end{align*}
which is a d-critical 
flip if $\sharp \mathbf{E}_{i0} <\sharp \mathbf{E}_{0i}$ for all $i \in V(Q)$.
It is a d-critical 
flop if $\sharp \mathbf{E}_{i0} =\sharp \mathbf{E}_{0i}$ for all $i \in V(Q)$.
\end{prop}
\begin{proof}
The assumption (\ref{M:inclu}) implies that 
$\pi_{Q}^{\pm}$ restricted to 
$M_{(Q, w)}^{\pm}$ factors through the 
closed immersion 
\begin{align*}
\mathrm{Rep}(Q)\sslash G \hookrightarrow 
\mathrm{Rep}(Q^{\dag})\sslash G
\end{align*}
induced by the zero section of the projection 
$\mathrm{Rep}(Q^{\dag}) \to \mathrm{Rep}(Q)$. 
Therefore we have the following diagram
\begin{align*}
\xymatrix{
M_{(Q^{\dag}, w)}^{\pm} \ar[d]_-{\pi_Q^{\pm}}
\ar@<-0.3ex>@{^{(}->}[r] & M_{Q^{\dag}}^{\pm} 
\ar[d]_-{\pi_Q^{\pm}} \ar[rd]^-{w^{\pm}} & \\
\mathrm{Rep}(Q)\sslash G \ar@<-0.3ex>@{^{(}->}[r] & 
\mathrm{Rep}(Q^{\dag})\sslash G \ar[r]_-{w_0} & \mathbb{C} 
}
\end{align*}
giving relative d-critical charts
of $M_{(Q^{\dag}, w)}^{\pm}$ (see Definition~\ref{def:A3}). 
Then the proposition follows from Proposition~\ref{prop:flip}. 
\end{proof}

\subsection{Relations with symmetric quiver representations}
We define the subquiver 
\begin{align*}
Q^{\dag}_0 \subset Q^{\dag}
\end{align*} 
defined by 
$V(Q^{\dag}_0)=V(Q^{\dag})$, 
the set of edges from $i \in V(Q^{\dag}_0)$
to $j \in V(Q^{\dag}_0)$ in $Q^{\dag}_0$ is the same 
as that in $Q^{\dag}$ except 
the case $i=0$ and $j \in V(Q)$: 
\begin{align*}
\{ 0 \to i \mbox{ in } Q^{\dag}_0 \} \cneq \mathbf{E}_{i0}
\subset \mathbf{E}_{0i}=
\{ 0 \to i \mbox{ in } Q^{\dag} \}. 
\end{align*}
The embedding (\ref{set:E}) realizes $Q_0^{\dag}$
as a subquiver of $Q^{\dag}$. 
From its construction, the quiver 
$Q_0^{\dag}$ is symmetric. 
Let $\mathrm{Rep}(Q_0^{\dag})$ be its 
representation space with dimension vector 
$((\dim V_i)_{i \in V(Q)}, 1)$. 
We have the following decomposition
as $\widetilde{G}$-representation
\begin{align*}
\mathrm{Rep}(Q^{\dag})=\mathrm{Rep}(Q^{\dag}_0)
\times \prod_{(0 \to i) \in \mathbf{E}_{0i} \setminus \mathbf{E}_{i0}}V_i. 
\end{align*}
Note that $\mathrm{Rep}(Q^{\dag}_0)$ is
a symmetric $G$-representation. 

\begin{lem}\label{lem:generic}
The element $\chi_0 \in M_{\mathbb{R}}^W$
defined by (\ref{chi0}) is generic 
with respect to the symmetric $G$-representation $\mathrm{Rep}(Q^{\dag}_0)$.
\end{lem}
\begin{proof}
Let $\beta_1, \ldots, \beta_{d} \in M$ for 
$d=\dim \mathrm{Rep}(Q_0^{\dag})$ be the $T$-weights of 
$\mathrm{Rep}(Q_0^{\dag})^{\vee}$. 
By~\cite[Proposition~2.1]{HLKSAM},
the genericity of $\chi_0$ is equivalent to that 
for any proper subspace $H \subsetneq M_{\mathbb{R}}$, 
there is a one parameter subgroup $\lambda \colon \mathbb{C}^{\ast} \to T$
such that $\langle \lambda, \beta_j \rangle =0$
for any $\beta_j \in H$ and 
$\langle \lambda, \chi_0 \rangle \neq 0$.   
Note that $T=\prod_{i \in V(Q)}T_i$
for the maximal torus $T_i \subset \GL(V_i)$. 
By choosing basis 
$\{e_{i1}, \ldots, e_{id_i}\}$
of $V_i$, 
where $d_i=\dim V_i$, 
we can write $M_{\mathbb{R}}$ as
\begin{align*}
M_{\mathbb{R}}=\bigoplus_{i \in V(Q)} \bigoplus_{1\le k\le d_i}
\mathbb{R}e_{ik}. 
\end{align*}
Then any non-zero $T$-character $\beta_j$ 
of $\mathrm{Rep}(Q_0^{\dag})^{\vee}$
is either of the form $\pm e_{ik}$ or
$e_{ik}-e_{i' k'}$. 
For a proper linear subspace $H \subsetneq M_{\mathbb{R}}$, 
let $\lambda\in N$ be 
the cocharacter defined by
\begin{align*}
\lambda=(\lambda_{ik})_{i \in V(Q), 1\le k \le d_i}, \ 
\lambda_{ik}=\left\{\begin{array}{cc}
0 & e_{ik} \in H \\
1 & e_{ik} \notin H
\end{array}   \right. 
\end{align*}
Then $\langle \lambda, \beta_j \rangle=0$ for any 
$\beta_j \in H$.
As $\chi_0=\sum_{i, k}e_{ik}$, we have 
\begin{align*}
\langle \lambda, \chi_0\rangle=
\sharp\{(i, k): i\in V(Q), 1\le k\le d_i\, e_{ik} \notin H\}
>0
\end{align*}
where the latter inequality holds as  
$H$ is a proper subspace. 
Therefore $\chi_0$ is generic. 
\end{proof}

We denote by 
\begin{align*}
\mathrm{Rep}(Q^{\dag}_0)^{\pm}_{\rm{ss}}
\subset \mathrm{Rep}(Q^{\dag}_0)
\end{align*}
the semistable locus with respect to the
$G$-action and $G$-linearizations 
$\oO(\chi_0^{\pm 1})$. 
\begin{lem}\label{lem:id:Rep}
We have the identity
\begin{align}\label{id:Rep}
\mathrm{Rep}(Q^{\dag})^{-}_{\rm{ss}}=\mathrm{Rep}(Q^{\dag}_0)^{-}_{\rm{ss}}
\times \prod_{(0 \to i) \in \mathbf{E}_{0i} \setminus \mathbf{E}_{i0}}V_i. 
\end{align}
\end{lem}
\begin{proof}
Let a $Q^{\dag}$-representation $E$ be represented by 
a point 
\begin{align*}
x=(x_e)_{e \in E(Q^{\dag})} \in \mathrm{Rep}(Q^{\dag}), \ 
x_e \colon V_{s(e)} \to V_{t(e)}
\end{align*}
where $V_0=\mathbb{C}$. 
We set
\begin{align*}
T=\bigoplus_{i \in V(Q)}V_i^{\vee}.
\end{align*}
Then $T$ has a  
$\mathbb{C}[Q]$-module structure 
given by $x_e^{\vee}$ for $e \in E(Q)$, 
which makes sense as $Q$ is symmetric.
By the proof of Proposition~\ref{prop:flip}, 
$E$ is semistable with respect to $\oO(\chi_0^{-1})$ if 
and only if it is $\theta^{-}$-semistable, 
where $\theta^-$ is given in (\ref{theta}).  
Then by~\cite[Lemma~7.10]{Toddbir}, we see that $E$ 
is $\theta^{-}$-semistable if and only if 
$T$ is 
 generated as a $\mathbb{C}[Q]$-module
by the images of 
the maps
\begin{align*}
x_e^{\vee} \colon \mathbb{C} \to 
V_i^{\vee}, \ e \in \mathbf{E}_{i0}.
\end{align*}
In particular, the $\theta^-$-stability 
does not impose any 
constraint on the maps 
$x_e \colon \mathbb{C} \to V_i$ for $e \in \mathbf{E}_{0i}$. 
The same statement applies to 
$\theta^-$-semistable 
$Q^{\dag}_0$-representations. 
Therefore (\ref{id:Rep}) holds. 
\end{proof}

\subsection{Window subcategories for quiver representations}
Let $W$ be a finite dimensional $G$-representation. 
We set
\begin{align*}
\mathrm{Rep}(Q^{\dag})_W
\cneq \mathrm{Rep}(Q^{\dag}) \times W \times W^{\vee}. 
\end{align*}
The $\widetilde{G}$-action on 
$\mathrm{Rep}(Q^{\dag})$ extends 
to the action on $\mathrm{Rep}(Q^{\dag})_W$ 
by 
\begin{align}\label{tilde:act}
(g, g') \cdot (x, u, u', v, v')=
(g(x), g(u), g'(u'), g(v), g'(v')). 
\end{align}
Here $(g, g') \in \widetilde{G}$ and 
\begin{align*}
(x, u, u', v, v') \in \mathrm{Rep}(Q) \times 
\left(\prod_{(0 \to i) \in \mathbf{E}_{0i}} V_i \right)
\times \left(\prod_{(i \to 0) \in \mathbf{E}_{i0}}V_i^{\vee}\right)
 \times W \times W^{\vee}.
\end{align*}  
We have the function $w^{\star}$ defined by
\begin{align}\label{w:tilde}
w^{\star} \colon 
\mathrm{Rep}(Q^{\dag})_W \to \mathbb{C}, \ 
w^{\star}(x, u, u', v, v')=w(x, u, u')
+\langle v, v' \rangle. 
\end{align}
Here $w$ is defined in (\ref{w:repQ}). 
Then $w^{\star}$ is semi-invariant 
for the character $\tau \colon \widetilde{G} \to \mathbb{C}^{\ast}$, 
$\tau((g_1, g_2))=g_1 g_2^{-1}$, which appeared in (\ref{exact:G}).   

We consider the induced $G$-action 
on $\mathrm{Rep}(Q^{\dag})_W$ and
the KN stratifications 
\begin{align}\label{KN:quiver}
\mathrm{Rep}(Q^{\dag})_W=
(\mathrm{Rep}(Q^{\dag})_W)_{\rm{ss}}^{\pm}
\sqcup S_{1}^{\pm} \sqcup S_{2}^{\pm} \sqcup \ldots. 
\end{align}
with respect to linearizations $\oO(\chi_0^{\pm 1})$, 
with centers $Z_{\alpha}^{\pm} \subset S_{\alpha}^{\pm}$
and one parameter subgroups 
$\lambda_{\alpha}^{\pm} \colon \mathbb{C}^{\ast} \to T$. 

We take 
\begin{align}\label{delta:chi}
\delta=\varepsilon \cdot \chi_0 \in M_{\mathbb{R}}^{W}, \ 
0<\varepsilon \ll 1
\end{align} 
where $\chi_0$ is the character (\ref{chi0}). 
We have 
the associated window subcategories
(see Definition~\ref{defi:window})
\begin{align}\label{window:quiver}
\cC_{\delta}^{\pm} \subset D_{\mathbb{C}^{\ast}}([\mathrm{Rep}(Q^{\dag})_W/G], 
w^{\star})
\end{align}
with respect to the KN stratifications (\ref{KN:quiver}). 
\begin{prop}\label{prop:window:quiver}
We have $\cC_{\delta}^- \subset \cC_{\delta}^+$. 
Moreover if $\mathbf{E}_{0i}=\mathbf{E}_{i0}$, we have 
$\cC_{\delta}^-=\cC_{\delta}^+$. 
\end{prop}
\begin{proof}
By Lemma~\ref{lem:id:Rep}, 
the proposition follows from Theorem~\ref{thm:magic}
and Corollary~\ref{cor:magic}. 
\end{proof}

\subsection{Complex analytic version}\label{subsec:analytic}
We consider the pull-backs of the 
following commutative diagram by an analytic open neighborhood
$0 \in \uU \subset \mathrm{Rep}(Q)\sslash G$
\begin{align}\label{dia:quiver:analytic}
\xymatrix{
\mathrm{Rep}(Q^{\dag})_W \ar[r] \ar[rd] \ar[rrd]^-{h_Q} & \mathrm{Rep}(Q^{\dag}) \ar[r] \ar[rd]^-{q_Q}
& \mathrm{Rep}(Q^{\dag})\sslash G \ar[d]^-{r_Q} \\
& \mathrm{Rep}(Q) \ar[r]_-{p_Q} & \mathrm{Rep}(Q)\sslash G. 
}
\end{align}
Here each horizontal arrows are either projections or quotient morphisms. 
Instead of the morphism (\ref{mor:s}), suppose that 
we are given a 
$G$-equivariant morphism of analytic 
vector bundles on $p_Q^{-1}(\uU)$
\begin{align}\label{mor:s2}
\vartheta \colon \left(\prod_{(0 \to i) \in \mathbf{E}_{0i}}V_i\right) \otimes \oO_{p_Q^{-1}(\uU)}
\to \left(\prod_{(i \to 0) \in \mathbf{E}_{i0}} V_i \right)
\otimes \oO_{p_Q^{-1}(\uU)}. 
\end{align}
Similarly to (\ref{w:repQ}), we have the $G$-invariant analytic map
\begin{align*}
w \colon q_Q^{-1}(\uU) \to \mathbb{C}
\end{align*}
which descends to the function $w^0$ on 
$r_{Q}^{-1}(\uU)$. 
By pulling back the diagram (\ref{dia:pm}) 
to $r_Q^{-1}(\uU) \subset \mathrm{Rep}(Q^{\dag})\sslash G$, 
we have the diagram
\begin{align}\notag
\xymatrix{
M_{Q^{\dag}, \uU}^{+} \ar[r]^-{\pi_Q^+} \ar[dr]_-{w^+} 
& r_Q^{-1}(\uU)
\ar[d]^-{w^0}& 
\ar[l]_-{\pi_Q^-} \ar[ld]^-{w^-} M_{Q^{\dag}, \uU}^- \\
& \mathbb{C}. &
}
\end{align}
We have the following analytic d-critical loci
\begin{align}\label{MU:dcrit}
M_{(Q^{\dag}, w), \uU}^{\pm}=\{dw^{\pm}=0\} \subset 
M_{Q^{\dag}, \uU}^{\pm}.
\end{align}
Similarly to Proposition~\ref{prop:dcrit}, we have the following: 
\begin{prop}\label{prop:dcrit2}
Suppose that the followings hold:
\begin{align}\notag
&M_{(Q^{\dag}, w), \uU}^+ \subset \left[\left(p_Q^{-1}(\uU) \times 
\prod_{(0 \to i) \in \mathbf{E}_{0i}}V_i \times \{0\}\right)/G  \right], \\ 
\notag
 & M_{(Q^{\dag}, w)}^- \subset \left[\left(p_Q^{-1}(\uU) \times \{0\} \times 
\prod_{(i \to 0) \in \mathbf{E}_{i0}}V_i^{\vee}\right) /G\right].
\end{align}
Then we have the diagram
\begin{align*}
\xymatrix{
M_{(Q^{\dag}, w), \uU}^{+} 
\ar[rd]_-{\pi_Q^+} & & \ar[ld]^-{\pi_Q^-} M_{(Q^{\dag}, w), \uU}^- \\
& \uU
}
\end{align*}
which is an analytic d-critical 
flip if $\sharp \mathbf{E}_{i0} <\sharp \mathbf{E}_{0i}$ for all $i \in V(Q)$.
It is an analytic d-critical 
flop if $\sharp \mathbf{E}_{i0} =\sharp \mathbf{E}_{0i}$ for all $i \in V(Q)$.
\end{prop}

We have 
the analytic map
\begin{align*}
w^{\star} \colon h_Q^{-1}(\uU) \to \mathbb{C}
\end{align*}
defined as in (\ref{w:tilde}), 
using (\ref{mor:s2}) instead of (\ref{mor:s}). 
The analytic open subset 
$h_Q^{-1}(\uU) \subset \mathrm{Rep}(Q^{\dag})_W$
is preserved by the $\widetilde{G}$-action, and
the above map $w^{\star}$ is $\tau$ semi-invariant. 
The derived factorization category 
$D_{\mathbb{C}^{\ast}}([h_Q^{-1}(\uU)/G], w^{\star})$
and the window subcategories
\begin{align}\label{window:analytic}
\cC_{\delta, \uU}^{\pm} \subset 
D_{\mathbb{C}^{\ast}}([h_Q^{-1}(\uU)/G], w^{\star})
\end{align}
are defined similarly to (\ref{window:quiver}), 
using $\delta$ as in (\ref{delta:chi})
and KN stratifications (\ref{KN:quiver}) restricted to 
$h_Q^{-1}(\uU)$. 
Then the same argument of Proposition~\ref{prop:window:quiver}
shows the following: 
\begin{prop}\label{prop:window:analytic}
We have $\cC_{\delta, \uU}^- \subset \cC_{\delta, \uU}^+$. 
Moreover if $\mathbf{E}_{0i}=\mathbf{E}_{i0}$, we have 
$\cC_{\delta, \uU}^-=\cC_{\delta, \uU}^+$. 
\end{prop}

\section{Geometry of moduli spaces of Thaddeus pairs}
In this section, we describe Thaddeus pair moduli spaces in terms of 
critical locus on some smooth stack. 
Then combined with the main result of~\cite{MR3811778}, 
we show that the diagram of Thaddus pair moduli space
and dual pair moduli space is a d-critical flip (flop), i.e. 
prove Theorem~\ref{thm:Tpair} (i). 

\subsection{Moduli spaces of semistable bundles}
Let $C$ be a smooth projective curve over $\mathbb{C}$
with genus $g$. 
For a vector bundle $E$ on $C$, 
its slope is defined by
\begin{align*}
\mu(E) \cneq \frac{\chi(E)}{\rank(E)} \in \mathbb{Q}. 
\end{align*}
Recall that $E$ is called \textit{(semi)stable} if 
for any subbundle $F \subsetneq E$, we have 
$\mu(F) <(\le) \mu(E)$. 

We denote by $\mM(r, d)$ the moduli stack of 
semistable bundles $E$ on $C$
satisfying the condition
\begin{align}\label{rank/chi2}
(\rank(E), \chi(E))=(r, d).
\end{align}
It is well-known that $\mM(r, d)$ is written as a global 
quotient stack, which we will review below. 
We take $m \gg 0$ and 
fix $p \in C$. 
Let $\mathbb{V}$ be a $\mathbb{C}$-vector space
such that 
\begin{align}\label{dim:V}
\dim \mathbb{V}=d+mr
\end{align} 
Let $\mathrm{Quot}(r, d)$ be the Grothendieck 
quot scheme parameterizing quotients
\begin{align}\label{quot}
\phi \colon 
\mathbb{V} \otimes \oO_C(-mp) \twoheadrightarrow E
\end{align}
such that $E$ satisfies (\ref{rank/chi2}). 
We have an open subset
\begin{align*}
\qQ \subset \mathrm{Quot}(r, d)
\end{align*}
corresponding to quotients (\ref{quot})
such that $E$ is 
a semistable bundle and
the induced map 
\begin{align*}
H^0(\phi(mp)) \colon \mathbb{V} \to H^0(E(mp))
\end{align*}
is an isomorphism. 
Note that both sides of the above map 
have the same dimension by (\ref{dim:V}). 
The natural $\GL(\mathbb{V})$-action on 
$\mathrm{Quot}(r, d)$ 
preserves the open subset $\qQ$, and 
we have 
\begin{align}\label{M=Q}
\mM(r, d)=[\qQ/\GL(\mathbb{V})]. 
\end{align}
Note that $\qQ$ is a smooth quasi-projective variety 
as $\mM(r, d)$ is a smooth stack. 

For $l\gg m$, let $L_l'$ be the line bundle on $\mathrm{Quot}(r, d)$
given by
\begin{align*}
L_{l}'|_{(\mathbb{V} \otimes \oO_C(-mp) \twoheadrightarrow E)}
=\det(H^0(C, E(lp))). 
\end{align*}
The line bundle $L_l'$ is a $\GL(\mathbb{V})$-linearized 
ample line bundle. 
Let 
$\chi_0$ be the determinant character of 
$\GL(\mathbb{V})$
\begin{align*}
\chi_0 \colon \GL(\mathbb{V}) \to \mathbb{C}^{\ast}, \ 
g \mapsto \det(g).
\end{align*}
We take a twist of $L_l'$ by a $\GL(\mathbb{V})$-character
to obtain the ample $\GL(\mathbb{V})$-linearized $\mathbb{Q}$-line bundle
on $\mathrm{Quot}(r, d)$
\begin{align}\label{Ll}
L_l \cneq L_l' \otimes \oO\left(-\frac{d+lr}{d+mr} \cdot \chi_0 \right). 
\end{align}
Then $\qQ$ is the semistable locus 
in the closure $\overline{\qQ}$ in $\mathrm{Quot}(r, d)$
with respect to the $\GL(\mathbb{V})$-linearization $L_l$
(see~\cite{MR1450870}). 
\begin{rmk}
The twist in $L_l$ is taken so that the $\GL(\mathbb{V})$-linearizion 
on $L_l$ is trivial on the diagonal torus 
$\mathbb{C}^{\ast} \subset \GL(\mathbb{V})$. Then 
the  
GIT $L_l$-stability for the $\GL(\mathbb{V})$-action 
is equivalent to 
the GIT $L_l'$-stability for the $\SL(\mathbb{V})$-action.
\end{rmk}
By taking the GIT quotient, we obtain the good moduli 
space for $\mM(r, d)$
\begin{align}\label{M:good}
\iota_M \colon \mM(r, d) \to M(r, d)=\qQ\sslash \GL(\mathbb{V}). 
\end{align}
The GIT quotient $M(r, d)$ is the 
coarse moduli space of $S$-equivalence classes of 
semistable bundles on $C$ satisfying (\ref{rank/chi2}). 
For a point $p \in M(r, d)$, 
by taking the associated graded 
of the Jordan-H\"older filtration, 
it corresponds to a polystable 
bundle $E$ of the form
\begin{align}\label{polystable}
E=\bigoplus_{i=1}^k V_i \otimes E_i. 
\end{align}
Here $V_i$ is a finite dimensional 
vector space, each $E_i$ is a stable bundle with 
$\mu(E_i)=\mu(E)$ and 
$E_i \not\cong E_j$ for $i\neq j$.

\subsection{Moduli spaces of Thaddeus pairs}
We recall the definition of Thaddeus pairs~\cite{MR1273268} and introduce 
the notion of dual Thaddeus pairs. 
Below we denote by $\mathrm{Bun}(C)$ the category 
of vector bundles on $C$. 
\begin{defi}\label{def:T}
(i) A pair 
\begin{align*}
(E, s), \ E \in \mathrm{Bun}(C), \ s \colon \oO_C \to E
\end{align*}
is called a Thaddeus pair if $E$ is semistable, 
there is no
non-zero subbundle $F \subsetneq E$ 
with $\mu(F)=\mu(E)$
such that $s$ factors through $F$. 

(ii) A pair 
\begin{align*}
(E', s'), \ E' \in \mathrm{Bun}(C), \ 
s' \colon E' \to \omega_C
\end{align*}
is called a dual 
Thaddeus pair if $E'$ is semistable, 
there is no non-zero subbundle
$F' \subsetneq E'$ 
with $\mu(F')=\mu(E')$ such that  
$s'|_{F'}=0$. 
\end{defi}
\begin{lem}\label{lem:dual}
Giving a dual Thaddeus pair $(E', s')$
with $[E'] \in \mM(r, d)$ is equivalent to giving 
a Thaddeus pair $(E, s)$ with $[E] \in \mM(r, -d)$. 
\end{lem}
\begin{proof}
It is easy to see that the following 
map gives a desired one to one correspondence 
\begin{align}\label{map:D}
\mathbb{D} \colon 
(E' \stackrel{s'}{\to} \omega_C) \mapsto 
(\oO_C \stackrel{s'^{\vee}}{\to} E'^{\vee} \otimes \omega_C). 
\end{align}
\end{proof}
We denote by $M^{\rm{T}}(r, d)$ the moduli space
of Thaddeus pairs $(E, s)$ with $[E] \in \mM(r, d)$. 
The moduli space $M^{\rm{T}}(r, d)$ is a smooth projective 
variety 
such that 
\begin{align}\label{dim:Tpair}
\dim M^{\rm{T}}(r, d)=\chi((\oO_C \to E), E)=d+r^2(g-1). 
\end{align}
Similarly we denote by $M^{\mathrm{DT}}(r, d)$ the moduli space 
of dual Thaddeus pairs $(E', s')$ such that 
$[E'] \in \mM(r, d)$. 
By Lemma~\ref{lem:dual}, the 
correspondence (\ref{map:D}) gives the isomorphism
\begin{align}\label{isom:dual}
\mathbb{D} \colon 
M^{\rm{DT}}(r, d) \stackrel{\cong}{\to} M^{\rm{T}}(r, -d). 
\end{align}
In particular, 
$M^{\rm{DT}}(r, d)$ is a smooth projective variety
 of dimension 
$-d+r^2(g-1)$. 
We have the natural morphisms
\begin{align}\label{diagram:MT}
\xymatrix{
	M^{\rm{T}}(r, d) \ar[r] \ar[dr]_-{\pi^+}& \mM(r, d) \ar[d]_-{\iota_M} & M^{\rm{DT}}(r, d) \ar[l] \ar[ld]^{\pi^-} \\
& M(r, d). &
}
\end{align}
Here $\pi^{\pm}$ are given by
$\pi^+(E, s)=\iota_M(E)$, $\pi^-(E', s')=\iota_M(E')$. 

\subsection{GIT descriptions of Thaddeus pair moduli spaces}
Let $\eE$ be the universal bundle
\begin{align*}
\eE \in \mathrm{Bun}(C \times \mM(r, d)). 
\end{align*}
Let $\mathrm{pr}_{\mM} \colon C \times \mM(r, d) \to \mM(r, d)$
be the projection. 
As $C$ is a curve, we can write 
$\dR \mathrm{pr}_{\mM\ast}\eE$ by a two term complex 
of vector bundles on $\mM(r, d)$
\begin{align}\label{resol:E}
\dR \mathrm{pr}_{\mM\ast}\eE=(\fF_0 \stackrel{\psi}{\to} \fF_1). 
\end{align}
Here $\fF_0$ is located in degree zero. 

By (\ref{M=Q}), the map $\psi$ is regarded as a
$\GL(\mathbb{V})$-equivariant morphism of 
$\GL(\mathbb{V})$-equivariant vector bundles 
$\widetilde{\psi} \colon \widetilde{\fF}_0 \to \widetilde{\fF}_1$
on $\qQ$. 
Since the diagonal torus $\mathbb{C}^{\ast} \subset \GL(\mathbb{V})$ 
acts on $\qQ$ trivially, 
it acts on fibers of $\widetilde{\fF}_{\bullet}$. 
\begin{lem}\label{F:wone}
We can take a resolution (\ref{resol:E})
so that 
the diagonal torus $\mathbb{C}^{\ast}$ acts 
on fibers of $\widetilde{\fF}_{\bullet}$ with weight one. 
\end{lem}
\begin{proof}
We have the decomposition of $\widetilde{\psi}$ into 
$\mathbb{C}^{\ast}$-weight spaces
\begin{align*}
\bigoplus_{n} \widetilde{\psi}_n \colon 
\bigoplus_{n} (\widetilde{\fF}_0)_n \to \bigoplus_{n} (\widetilde{\fF}_1)_n
\end{align*}
where the diagonal torus $\mathbb{C}^{\ast} \subset \GL(\mathbb{V})$
acts on 
$(\widetilde{\fF}_{\bullet})_n$ with weight $n$. 
Since the diagonal $\mathbb{C}^{\ast}$-actions on 
$R^i \mathrm{pr}_{\mM\ast}\eE|_{[E]}=H^i(C, E)$ are
given by the scaling action for $\mathbb{C}^{\ast} \subset \Aut(E)$, 
they have weight one. 
Therefore
the morphism 
$\widetilde{\psi}_n$ is an isomorphism for all $n\neq 1$.
Then we can 
replace the RHS of (\ref{resol:E}) by a quasi-isomorphism 
so that  
$(\widetilde{\fF}_{\bullet})_n=0$ for all $n\neq 1$. 
\end{proof}
Below, we take 
a resolution (\ref{resol:E}) as in 
Lemma~\ref{F:wone}. 
By setting $\fF^{-i}=\fF_i^{\vee}$,
the Grothendieck 
duality implies  
\begin{align*}
\dR \mathrm{pr}_{\mM\ast}(\eE^{\vee}\boxtimes \omega_C)
\cong
(\dR \mathrm{pr}_{\mM\ast}\eE)^{\vee}[-1]
=(\fF^{-1} \stackrel{\psi^{\vee}}{\to} \fF^0).
\end{align*}
Here $\fF^{-1}$ is located in degree zero.
The map $\psi^{\vee}$ is 
regarded as 
a $\GL(\mathbb{V})$-equivariant 
morphism
$(\widetilde{\fF}^{-1} \stackrel{\widetilde{\psi}^{\vee}}{\to} 
\widetilde{\fF}^0)$
on $\qQ$, where 
for $\widetilde{\fF}^{-i}=(\widetilde{\fF}_i)^{\vee}$.  
By regarding $\fF_{\bullet}$, $\fF^{\bullet}$
as total spaces of vector bundles on $\mM(r, d)$, we have the 
following diagrams of stacks over $\mM(r, d)$
\begin{align}\label{dia:stacks}
\xymatrix{
& 
\fF_0 \times_{\mM(r, d)} \fF_1 \ar[r]  \ar@<1.0ex>[d] & \fF_1 \ar[d] \\
\zZ \cneq (\psi=0) \ar@<-0.3ex>@{^{(}->}[r] & 
\fF_0 \ar[r]\ar@<1.0ex>[u]^-{\psi} & \mM(r, d),
}\\
\notag
\xymatrix{
& 
\fF^{-1} \times_{\mM(r, d)} \fF^0 \ar[r]  \ar@<1.0ex>[d] & \fF^0 \ar[d] \\
\zZ^{\vee} \cneq (\psi^{\vee}=0) \ar@<-0.3ex>@{^{(}->}[r] & 
\fF^{-1} \ar[r]\ar@<1.0ex>[u]^-{\psi^{\vee}} & \mM(r, d).
}
\end{align}
In the above diagrams, $\psi$ and $\psi^{\vee}$ are regarded as sections of 
vector bundles on $\fF_0$ and $\fF^{-1}$. 
Let $\fF_0^{\ast} \subset \fF_0$, 
$\fF^{-1\ast} \subset \fF^{-1}$ be the complement of the
zero sections of the projections 
$\fF_0 \to \mM(r, d)$, $\fF^{-1} \to \mM(r, d)$ respectively. 
\begin{lem}\label{lem:Zpair}
(i) The stack $\zZ$ is isomorphic to the stack of 
pairs $(E, s)$ for $[E] \in \mM(r, d)$
and a morphism $s \colon \oO_C \to E$. 
Its open subset $\zZ \cap \fF_0^{\ast}$ is smooth 
such that
\begin{align*}
\dim (\zZ \cap \fF_0^{\ast})=\dim M^{\rm{T}}(r, d)=
d+r^2(g-1). 
\end{align*}

(ii) The stack $\zZ^{\vee}$ is isomorphic to the 
stack of pairs $(E', s')$
for $[E'] \in \mM(r, d)$
and a morphism $s' \colon E' \to \omega_C$. 
Its open subset $\zZ^{\vee} \cap \fF^{-1\ast}$ is smooth 
such that
\begin{align*}
\dim (\zZ^{\vee} \cap \fF^{-1\ast})=\dim M^{\rm{DT}}(r, d)=
-d+r^2(g-1). 
\end{align*}
\end{lem}
\begin{proof}
By noting that $\Ker(\psi)=\mathrm{pr}_{\mM\ast}\eE$
and $\Ker(\psi^{\vee})=\mathrm{pr}_{\mM\ast}(\eE^{\vee}\boxtimes \omega_C)$, 
the first statements in (i), (ii) are straightforward to prove. 
As for the second statements, we only show (i) as (ii) is similar. 

By the first statement, the stack $\zZ \cap \fF_0^{\ast}$ is the 
stack of pairs $(E, s)$ for $[E] \in \mM(r, d)$
and $s \colon \oO_C \to E$ is a non-zero morphism. 
By the deformation-obstruction theory of pairs, it is enough 
to show the vanishing 
\begin{align}\label{Gamma:vanish}
\hH^1(\dR \Gamma((\oO_C \to E), E))=0
\end{align}
where 
$\oO_C$ is located in degree zero. 
By applying $\RHom(-, E)$ to the distinguished triangle
\begin{align*}
E[-1] \to (\oO_C \to E) \to \oO_C
\end{align*}
we obtain the exact sequence
\begin{align}\label{exact:E}
\Ext^1(E, E) \to H^1(C, E) \to \hH^1(\dR \Gamma((\oO_C \to E), E))
\to 0. 
\end{align}
The first arrow is Serre dual to 
the map \begin{align*}
\Hom(E, \omega_C) \to \Hom(E, E \otimes \omega_C)
\end{align*}
induced by the map $s \colon \oO_C \to E$. 
As $s$ is non-zero, the above map is injective, 
so the first arrow in (\ref{exact:E})
is surjective. 
Therefore (\ref{Gamma:vanish}) holds. 
\end{proof}

By defining $\widetilde{\zZ}$, $\widetilde{\zZ}^{\vee}$ 
similarly as 
\begin{align}\label{dia:tilde}
\xymatrix{
\widetilde{\zZ} \cneq (\widetilde{\psi}=0) \ar@<-0.3ex>@{^{(}->}[r] \ar[rd] & \widetilde{\fF}_0 \ar[d]^-{p_0} \\
& \qQ, 
}
\quad 
\xymatrix{
\widetilde{\zZ}^{\vee} \cneq (\widetilde{\psi}^{\vee}=0) \ar@<-0.3ex>@{^{(}->}[r] \ar[rd] & \widetilde{\fF}^{-1} \ar[d]^-{p^{-1}} \\
& \qQ, 
}
\end{align}
the stacks 
$\zZ$, $\zZ^{\vee}$ are written as 
\begin{align*}
\zZ=[\widetilde{\zZ}/\GL(\mathbb{V})], \ 
\zZ^{\vee}=[\widetilde{\zZ}^{\vee}/\GL(\mathbb{V})]. 
\end{align*}
Note that $\widetilde{\zZ}$, $\widetilde{\zZ}^{\vee}$
parametrize diagrams
\begin{align}\label{dia:Z}
\xymatrix{
 & \mathbb{V} \otimes \oO_C(-mp) \ar@{->>}[d]\\
 \oO_C \ar[r]^-{s} & E, 
} \quad 
\xymatrix{
 \mathbb{V} \otimes \oO_C(-mp) \ar@{->>}[d] & \\
 E' \ar[r]^-{s'} & \omega_C
}
\end{align}
such that $[E] \in \mM(r, d)$, $[E'] \in \mM(r, d)$
respectively. 
In the diagram (\ref{dia:tilde}), 
we define the following $\GL(\mathbb{V})$-equivariant 
$\mathbb{Q}$-line bundles 
\begin{align}\label{line:Lpm}
L_l^+&=p_0^{\ast}L_l \otimes \oO(\chi_0^{\varepsilon}) \in \Pic_{\GL(\mathbb{V})}(\widetilde{\fF}_0)_{\mathbb{Q}}, \\
\notag
L_l^-&=p^{-1\ast}L_l \otimes \oO(\chi_0^{-\varepsilon}) \in \Pic_{\GL(\mathbb{V})}(\widetilde{\fF}^{-1})_{\mathbb{Q}}
\end{align}
for a rational number $0<\varepsilon \ll 1$. 
\begin{lem}
(i) A left diagram of (\ref{dia:Z}) is $L_l^+$-semistable if 
and only if the pair $(\oO_C \stackrel{s}{\to}E)$ is a Thaddeus pair. 

(ii) A right diagram of (\ref{dia:Z}) is $L_l^-$-semistable if and only 
if the pair $(E' \stackrel{s'}{\to} \omega_C)$ is 
a dual Thaddeus pair.  
\end{lem}
\begin{proof}
The lemma follows from a standard argument
applying the Hilbert-Mumford criterion. 
See~\cite{MR2869309, MR3187656, MR3743369} for related arguments. 
\end{proof}
By the above lemma, 
we have
the following descriptions of Thaddeus pair moduli spaces
\begin{align}\label{Tpair:GIT}
&(\fF_0)_{{L_l^+}\rm{\mathchar`-ss}}
\cap \zZ=\zZ_{{L_l^+}\rm{\mathchar`-ss}}=M^{\rm{T}}(r, d), \\
\notag
&(\fF^{-1})_{{L_l^-}\rm{\mathchar`-ss}} \cap \zZ^{\vee}
=\zZ_{{L_l^-}\rm{\mathchar`-ss}}^{\vee}=M^{\rm{DT}}(r, d). 
\end{align}

\subsection{Descriptions by critical locus}
Let $\widetilde{\yY}$ be the variety defined by
\begin{align*}
\widetilde{\yY} \cneq \widetilde{\fF}_0 \times_{\qQ} \widetilde{\fF}^{-1}
\end{align*}
and define the stack $\yY$ to be 
\begin{align*}
\yY \cneq \fF_0 \times_{\mM(r, d)}\fF^{-1}
=[\widetilde{\yY}/\GL(\mathbb{V})]. 
\end{align*}
We define the function 
$w$ on $\widetilde{\yY}$ to be
\begin{align}\label{funct:w}
w \colon \widetilde{\fF}_0 \times_{\qQ} \widetilde{\fF}^{-1}
\to \mathbb{A}^1, \ 
(x, u, u') \mapsto \langle \widetilde{\psi}|_{x}(u), u' \rangle. 
\end{align}
Here $x \in \qQ$, $u \in \widetilde{\fF}_0|_{x}$ and $u' \in 
\widetilde{\fF}^{-1}|_{x}$. 
Then $w$ is $\GL(\mathbb{V})$-invariant, 
so it descends to the morphism
\begin{align*}
w \colon \yY \to \mathbb{A}^1. 
\end{align*}
We define $\widetilde{\wW}$, $\wW$ to be
\begin{align*}
\widetilde{\wW}  \cneq \{dw=0\} \subset 
\widetilde{\yY}, \
\wW \cneq [\widetilde{\wW}/\GL(\mathbb{V})]=\{dw=0\}  \subset \yY.
\end{align*}
Let $p_{\yY} \colon \widetilde{\yY} \to \qQ$ be the 
projection, and define the following 
$\GL(\mathbb{V})$-linearized 
ample 
$\mathbb{Q}$-line 
bundles on $\widetilde{\yY}$
\begin{align*}
L_l^{\pm} \cneq p_{\yY}^{\ast}L_l \otimes 
\oO(\chi_0^{\pm \varepsilon}), \ 
\varepsilon \in \mathbb{Q}, \ 
0<\varepsilon \ll 1. 
\end{align*}
The above $\mathbb{Q}$-line bundles restrict to
the $\mathbb{Q}$-line bundles (\ref{line:Lpm}) by the 
closed immersions
\begin{align*}
i_0 \colon \fF_0 \hookrightarrow \yY, \ 
i^{-1} \colon \fF^{-1} \hookrightarrow \yY
\end{align*}
given by 
zero sections of the projections 
$\yY \to \fF_0$, $\yY \to \fF^{-1}$ respectively. 
By (\ref{Tpair:GIT}), we have the embeddings
\begin{align}\label{MT:emb}
&M^{\rm{T}}(r, d) =\zZ_{L_l^+\rm{\mathchar`-ss}} \subset
\fF_0 \stackrel{i_0}{\hookrightarrow} \yY, \\
\notag 
&M^{\rm{DT}}(r, d) =\zZ_{L_l^-\rm{\mathchar`-ss}}^{\vee} \subset
\fF^{-1} \stackrel{i^{-1}}{\hookrightarrow} \yY.
\end{align}
\begin{prop}\label{prop:MW}
The embeddings (\ref{MT:emb}) induce the
isomorphisms
\begin{align*}
M^{\rm{T}}(r, d) \stackrel{\cong}{\to} \wW_{L_l^+\rm{\mathchar`-ss}}, \ 
M^{\rm{DT}}(r, d) \stackrel{\cong}{\to} \wW_{L_l^-\rm{\mathchar`-ss}}.
\end{align*}
\end{prop}
\begin{proof}
We only show the 
first isomorphism
as the latter is similarly proved. 
By the definition of $w$, 
the image of $\zZ \subset \fF_0 \stackrel{i_0}{\hookrightarrow} \yY$
lies in the critical locus $\wW \subset \yY$. 
As $\zZ \subset \wW$ is a closed immersion, 
we have 
\begin{align*}
M^{\rm{T}}(r, d)=\zZ_{L_l^+\rm{\mathchar`-ss}} \hookrightarrow
\wW_{L_l^+\rm{\mathchar`-ss}}.
\end{align*}
It is enough to show that 
$\wW_{L_l^+\rm{\mathchar`-ss}} \subset \zZ$, as if this is the case
we have
\begin{align*}
\wW_{L_l^+\rm{\mathchar`-ss}} \subset \zZ_{L_l^+\rm{\mathchar`-ss}}
=M^{\rm{T}}(r, d). 
\end{align*}
Below we show that $\wW_{L_l^+\rm{\mathchar`-ss}} \subset \zZ$. 
By the $L_l^+$-stability, 
we have 
\begin{align}\label{L+:stab}
\wW_{L_l^+\rm{\mathchar`-ss}}
\subset (\fF_0^{\ast} \times_{\mM(r, d)}\fF^{-1}) \cap \wW. 
\end{align}
Indeed for 
the anti-diagonal 
$\lambda \colon \mathbb{C}^{\ast} \to \GL(\mathbb{V})$
and 
a point 
$x \in \widetilde{\fF}^{-1} \stackrel{i^{-1}}{\hookrightarrow} 
\widetilde{\yY}$, 
as $\lambda$ acts on $\widetilde{\fF}^{-1}$ with weight one, 
we have
\begin{align*}
y \cneq \lim_{t \to 0} \lambda(t)(x) \in 
\qQ \subset \widetilde{\yY}
\end{align*}
where the latter is the zero section of the 
projection $\widetilde{\yY} \to \qQ$. 
Then we have 
\begin{align*}
\mathrm{weight}_{\lambda}(L_l^+|_{y})=
\varepsilon \cdot \langle \lambda, \chi_0\rangle=
-\varepsilon \cdot \dim \mathbb{V}<0.
\end{align*}
Therefore $x$ is not $L_l^+$-semistable, and 
the inclusion (\ref{L+:stab}) holds. 

By Lemma~\ref{lem:Zpair},
the map 
\begin{align*}
\psi \colon \fF_0^{\ast} \to 
\fF_0^{\ast} \times_{\mM(r, d)} \fF_1
\end{align*}
is a regular section 
of the vector bundle 
$\fF_0^{\ast} \times_{\mM(r, d)}\fF_1 \to \fF_0^{\ast}$, 
whose zero locus 
is smooth. 
Therefore by Lemma~\ref{lem:Z}, we have 
\begin{align*}
(\fF_0^{\ast} \times_{\mM(r, d)}\fF^{-1}) \cap \wW=
\zZ \cap \fF_0^{\ast} \subset \zZ. 
\end{align*}
\end{proof}

\subsection{Analytic local descriptions}
For a point $p \in M(r, d)$, corresponding to 
a polystable bundle $E$ as in (\ref{polystable}), 
we have 
the associated \textit{Ext-quiver} $Q_{E}$. 
Its vertex set is given by $\{1, 2, \ldots, k\}$, and 
the number of edges from $i$ to $j$ is 
\begin{align*}
\sharp(i \to j)=\dim \Ext^1(E_i, E_j)=\left\{ \begin{array}{ll}
(g-1)r_i r_j, & i\neq j, \\
(g-1)r_i^2 +1, & i=j.
\end{array} \right. 
\end{align*}
Here $r_i=\rank(E_i)$. 
Note that $\sharp(i \to j)=\sharp(j \to i)$, i.e. 
$Q_E$ is a symmetric quiver.
Then we have
\begin{align*}
\Ext^1(E, E)=\mathrm{Rep}(Q_E) \cneq 
\prod_{i \to j } \Hom(V_i, V_j).
\end{align*}
Let $G_E$ be the algebraic group defined by 
\begin{align*}
G_E=\Aut(E)=\prod_{i=1}^k \GL(V_i). 
\end{align*}
The quotient of $\mathrm{Rep}(Q_E)$ by the 
conjugate action of $G_E$ is the moduli 
stack of $Q_E$-representations with dimension vector 
$(\dim V_i)_{1\le i\le k}$. 
We have the morphism to the good moduli space
\begin{align*}
\iota_Q \colon [\mathrm{Rep}(Q_E)/G_E]
\to \mathrm{Rep}(Q_E) \sslash G_E.
\end{align*}
The following is a special case of the main result of~\cite{MR3811778}. 
\begin{thm}\emph{(\cite{MR3811778})}\label{thm:analytic}
There exist analytic open neighborhoods
$p \in \vV \subset M(r, d)$, 
$0 \in \uU \subset \mathrm{Rep}(Q_E) \sslash G_E$
and commutative isomorphisms
\begin{align*}
\xymatrix{
\iota_Q^{-1}(\uU) \ar[r]^-{\cong}_-{\alpha} \ar[d]^-{\iota_Q}& \iota_M^{-1}(\vV) 
\ar[d]^-{\iota_M} \\
\uU \ar[r]^-{\cong} & \vV. 
}
\end{align*}
Here the map $\alpha$ sends $0 \in \iota_Q^{-1}(\uU)$ to 
$[E] \in \iota_M^{-1}(\vV)$. 
\end{thm}
Let $p_Q$ be the quotient morphism 
\begin{align*}
p_Q \colon \mathrm{Rep}(Q_E)
\to \mathrm{Rep}(Q_E) \sslash G_E. 
\end{align*}
We take $0 \in \uU$ to be a Stein analytic open 
neighborhood, so 
that $p_Q^{-1}(\uU)$ is also Stein. 
Then we have the following lemma: 
\begin{lem}\label{lem:decompose}
For the isomorphism $\alpha$ in Theorem~\ref{thm:analytic},
the complex 
\begin{align*}
\alpha^{\ast}(\fF_0 \stackrel{\psi}{\to} \fF_1)|_{\iota_M^{-1}(\vV)}
\end{align*}
is isomorphic to the direct sum
\begin{align*}
(H^0(E) \otimes \oO_{p_Q^{-1}(\uU)}
 \stackrel{\vartheta}{\to} H^1(E) \otimes \oO_{p_Q^{-1}(\uU)})
\oplus (W \otimes \oO_{p_Q^{-1}(\uU)}
 \stackrel{\mathrm{id}}{\to} W \otimes \oO_{p_Q^{-1}(\uU)})
\end{align*}
for some 
$G_E$-equivariant morphism $\vartheta$ and 
a 
finite dimensional $G_E$-representation $W$. 
\end{lem}
\begin{proof}
We can write  
\begin{align*}
\alpha^{\ast}(\fF_0 \stackrel{\psi}{\to} \fF_1)|_{\iota_M^{-1}(\vV)}
\cong (R_0 \otimes \oO_{p_Q^{-1}(\uU)}
\stackrel{\phi}{\to} R_1 \otimes \oO_{p_Q^{-1}(\uU)})
\end{align*}
for some finite dimensional 
$G_E$-representations $R_0$, $R_1$
and a $G_E$-equivariant morphism $\phi$. 
As the map $\alpha$ sends $0 \in p_Q^{-1}(\uU)$ to 
$[E] \in \iota_M^{-1}(\vV)$, 
we have 
\begin{align*}
\Ker(\phi|_{0})=H^0(E), \ \Cok(\phi|_{0})=H^1(E)
\end{align*}
as $G_E$-representations. 
As $G_E$ is reductive, 
we have
\begin{align*}
R_0=H^0(E) \oplus W, \ 
R_1=H^1(E) \oplus W
\end{align*}
for some finite dimensional 
$G_E$-representation $W$, and 
the map $\phi|_0$ is written as
\begin{align*}
(H^0(E) \stackrel{0}{\to} H^1(E)) \oplus 
(W \stackrel{\id}{\to} W).
\end{align*}
Then $\phi$ is written as
\begin{align*}
\phi=\left(\begin{array}{cc}
\phi_{11} & \phi_{12} \\
\phi_{21} & \phi_{22}
\end{array} \right) \colon 
(H^0(E) \oplus W) \otimes \oO_{p_Q^{-1}(\uU)}
\to (H^1(E) \oplus W) \otimes \oO_{p_Q^{-1}(\uU)}. 
\end{align*}
Here $\phi_{ij}$ are $G_E$-equivariant morphisms
\begin{align*}
\phi_{ij} \colon U_i \otimes \oO_{p_Q^{-1}(\uU)}
\to U_j' \otimes \oO_{p_Q^{-1}(\uU)}
\end{align*}
for $U_1=H^0(E)$, $U_1'=H^1(E)$, 
$U_2=U_2'=W$, such that 
\begin{align*}
\phi_{ij}|_{0}=0, \ (i, j) \neq (2, 2), \ 
\phi_{22}|_{0}=\id. 
\end{align*}
By shrinking $\uU$ if necessary, we may assume that 
$\phi_{22}$ is an isomorphism. 
Then by replacing $\phi$ by an automorphism of 
$W \otimes \oO_{p_Q^{-1}(\uU)}$, we 
may assume that $\phi_{22}=\id$. 
Then by the following replacement of $\phi$
by automorphisms of both sides
\begin{align*}
\left(\begin{array}{cc}
\id & -\phi_{12} \\
0 & \id
\end{array} \right)
\left(\begin{array}{cc}
\phi_{11} & \phi_{12} \\
\phi_{21} & \phi_{22}
\end{array} \right)
\left(\begin{array}{cc}
1 & 0 \\
-\phi_{21} & 1
\end{array} \right)
\end{align*}
we may assume that $\phi_{12}=0$, $\phi_{21}=0$. 
By setting $\vartheta=\phi_{11}$, we obtain the lemma. 
\end{proof}

In what follows, we assume that $d\ge 0$. 
Then we have 
\begin{align*}
h^0(E_{i}) - h^1(E_i)
=\chi(E_i)=r_i \cdot \frac{d}{r} \ge 0
\end{align*}
for 
$1\le i\le k$. For each $1\le i\le k$, let 
 us fix finite subsets and an injective map
\begin{align*}
\mathbf{E}_{0i} \subset H^0(E_i), \ \mathbf{E}_{i0} \subset H^1(E_i), \
\mathbf{E}_{i0} \hookrightarrow \mathbf{E}_{0i}
\end{align*}
where $\mathbf{E}_{0i}$, $\mathbf{E}_{i0}$ give 
basis of $H^0(E_i)$, $H^1(E_i)$
respectively. 
Let $Q_E^{\dag}$ be the extended quiver as in Section~\ref{subsec:quiver}, 
constructed from $Q=Q_{E}$ and $\mathbf{E}_{0i}$, $\mathbf{E}_{i0}$
as above. 
Then we have
\begin{align*}
H^0(E)=\prod_{(0 \to i) \in \mathbf{E}_{0i}}V_i, \quad 
H^1(E)=\prod_{(0 \to i) \in \mathbf{E}_{i0}}V_i. 
\end{align*}
Therefore in the notation of Section~\ref{subsec:quiver}, we have 
\begin{align*}
&\mathrm{Rep}(Q_E^{\dag})=\mathrm{Rep}(Q_E)
\times H^0(E) \times H^1(E)^{\vee}, \\
&\mathrm{Rep}(Q_E^{\dag})_W=\mathrm{Rep}(Q_E)
\times H^0(E) \times H^1(E)^{\vee} \times W \times W^{\vee}.
\end{align*} 
Here $W$ is a $G_E$-representation in Lemma~\ref{lem:decompose}. 
As in the diagram (\ref{dia:quiver:analytic}), 
we have the map
\begin{align*}
h_{Q} \colon \mathrm{Rep}(Q_E^{\dag})_W
\to \mathrm{Rep}(Q_E)\sslash G_E
\end{align*}
by composing the projection and the quotient map. 
So we have
\begin{align}\label{stack:analytic}
[h_Q^{-1}(\uU)/G_E]
=\left[ \left(p_Q^{-1}(\uU) \times H^0(E) \times H^1(E)^{\vee}
\times W \times W^{\vee}\right)/G_E \right].
\end{align}
Let $\pi_{\yY}$ be the composition of the natural maps
\begin{align*}
\pi_{\yY} \colon 
\yY \to \mM(r, d) \to M(r, d).
\end{align*}
Then Theorem~\ref{thm:analytic}
together with Lemma~\ref{lem:decompose}
imply that 
we have the following commutative isomorphisms
\begin{align}\label{dia:pair:analytic}
\xymatrix{
[h_{Q}^{-1}(\uU)/G_E] 
\ar[r]_-{\xi}^-{\cong} \ar[d]^-{h_Q} & \pi_{\yY}^{-1}(\vV) 
\ar[d]^-{\pi_{\yY}} \\
\uU \ar[r]^-{\cong} & \vV.
}
\end{align}
Under the isomorphism $\xi$, 
the line bundles $L_l^{\pm}$ on $\yY$ 
are pulled back
to the 
$G_E$-equivariant $\mathbb{Q}$-line bundles
\begin{align}\label{line:pullback}
\oO(\chi_0^{\pm \varepsilon})
\in \Pic_{G_E}(h_Q^{-1}(\uU))_{\mathbb{Q}}
\end{align}
where $\chi_0 \colon G_E \to \mathbb{C}^{\ast}$ is given by (\ref{chi0})
for $Q=Q_E$. 
Moreover  
the composition
\begin{align*}
w^{\star} \colon 
[h_Q^{-1}(\uU)/G_E] \stackrel{\xi}{\to} \pi_{\yY}^{-1}(\vV)
\hookrightarrow \yY \stackrel{w}{\to} \mathbb{A}^1
\end{align*}
is given by 
\begin{align}\label{wQ}
w^{\star}(x, u, u', v, v')=
\langle \vartheta|_{x}(u), u' \rangle +\langle v, v' \rangle.
\end{align}
Here $x \in p_Q^{-1}(\uU)$, $u \in H^0(E)$, $u' \in H^1(E)^{\vee}$, 
$v \in W$, $v' \in W^{\vee}$, and $\vartheta$ is 
given in Lemma~\ref{lem:decompose}. 
Then we show the following: 
\begin{thm}\label{thm:dflip}
For $d>0$, the diagram 
\begin{align*}
\xymatrix{
	M^T(r, d)  \ar[dr]_-{\pi^+}&  & M^{\rm{DT}}(r, d)  \ar[ld]^{\pi^-} \\
& M(r, d). &
}
\end{align*}
is an analytic d-critical flip. 
For $d=0$, it is an analytic d-critical flop. 
\end{thm}
\begin{proof}
Let $p \in M(r, d)$ corresponds to a polystable 
bundle $E$ as in (\ref{polystable}), 
and take $p \in \vV \subset M(r, d)$, $0 \in \uU \subset 
\mathrm{Rep}(Q_E)\sslash G$
as in Theorem~\ref{thm:analytic}. 
Then we have the diagram (\ref{dia:pair:analytic}) as discussed above. 
The critical locus of the function (\ref{wQ})
is the same as the critical locus of 
\begin{align*}
w_Q \colon \left[ \left(p_Q^{-1}(\uU) \times H^0(E) \times H^1(E)^{\vee} \right)/G_E \right] \to \mathbb{C} 
\end{align*}
defined by 
$w_Q(x, u, u')=\langle \vartheta|_{x}(u), u' \rangle$.
Together with the compatibility of linearizations (\ref{line:pullback}), 
in the notation of Section~\ref{subsec:analytic}
we have the commutative isomorphisms
\begin{align}\label{dia:T:dflip}
\xymatrix{
M_{(Q_{E}^{\dag}, w_Q), \uU}^{\pm} \ar[r]^-{\cong} \ar[d]^-{\pi_{Q}^{\pm}} & 
\wW_{L_l^{\pm}\rm{\mathchar`-ss}} \cap \pi_{\yY}^{-1}(\vV) 
\ar[d]^-{\pi_{\yY}} &
(\pi^{\pm})^{-1}(\vV) \ar[l]_-{\cong} \ar[d]^-{\pi^{\pm}}\\
\uU \ar[r]^-{\cong} & \vV  & \vV \ar[l]_-{\id}. 
}
\end{align}
Here the right diagram follows from 
Proposition~\ref{prop:MW}, and 
\begin{align*}
M_{(Q_{E}^{\dag}, w_Q), \uU}^{\pm}=\left[\left((p_Q^{-1}(\uU) \times H^0(E) \times H^1(E)^{\vee})_{\oO(\chi_0^{\pm \varepsilon})\rm{\mathchar`-ss}}\cap \{dw_{Q}=0\}\right)/G_E \right]
\end{align*}
as given in (\ref{MU:dcrit}) for $Q^{\dag}=Q_E^{\dag}$. 
The above isomorphisms imply that 
the assumption in Proposition~\ref{prop:dcrit2}
is satisfied, 
so the result follows from Proposition~\ref{prop:dcrit2}. 
\end{proof}

\section{Derived functors between Thaddeus pair moduli spaces}
In this section, we combine the arguments so far and 
finish the proof of Theorem~\ref{thm:Tpair} (ii). 
We also investigate the fully-faithful functor $\Phi_M$
on the stable locus, and give an explicit description of the 
kernel object. 

\subsection{Window subcategories for Thaddeus pair moduli spaces}
Let $\widetilde{\GL}(\mathbb{V})$ be the algebraic group defined 
in (\ref{Gtilde}) for $G=\GL(\mathbb{V})$
and the diagonal torus $\mathbb{C}^{\ast} \subset \GL(\mathbb{V})$. 
Then $\widetilde{\GL}(\mathbb{V})$ acts on $\widetilde{\yY}$
by
\begin{align*}
(g_1, g_2) \cdot (x, u, u')&=(g_1(x), g_1(u), g_2(u')) \\
&=(g_2(x), g_1(u), g_2(u'))
\end{align*}
for $(g_1, g_2) \in \widetilde{\GL}(\mathbb{V})$, 
 $x \in \qQ$, $u \in \widetilde{\fF}_0|_{x}$ and $u' \in 
\widetilde{\fF}^{-1}|_{x}$.
Here the latter equality holds as 
the diagonal torus $\mathbb{C}^{\ast} \subset \GL(\mathbb{V})$ 
acts on $\qQ$ trivially. 
The function $w$ in (\ref{funct:w}) is $\tau$ semi-invariant, 
where $\tau \colon \widetilde{\GL}(\mathbb{V}) \to \mathbb{C}^{\ast}$ 
is the character $(g_1, g_2) \mapsto g_1 g_2^{-1}$
in the exact sequence (\ref{exact:G}). 

We have the 
KN stratifications
\begin{align}\label{KN:Y}
\widetilde{\yY}=\widetilde{\yY}_{L_l^{\pm}\rm{\mathchar`-ss}}
\sqcup S_1^{\pm} \sqcup S_2^{\pm} \sqcup \cdots, \ 
\yY=\yY_{L_l^{\pm}\rm{\mathchar`-ss}}
\sqcup \sS_1^{\pm} \sqcup \sS_2^{\pm} \sqcup \cdots
\end{align}
with respect to the 
$\GL(\mathbb{V})$-linearizations 
$L_l^{\pm}$, and the 
associated 
window subcategories (see Definition~\ref{defi:window})
\begin{align*}
\cC_{\yY}^{\pm} \subset D_{\mathbb{C}^{\ast}}(\yY, 
w). 
\end{align*}
\begin{prop}\label{prop:equiv:CY}
We have equivalences of triangulated categories
\begin{align}\label{equiv:CY}
\cC_{\yY}^+ \stackrel{\sim}{\to} D^b(M^{\rm{T}}(r, d)), \ 
\cC_{\yY}^- \stackrel{\sim}{\to} D^b(M^{\rm{DT}}(r, d)). 
\end{align}
\end{prop}
\begin{proof}
By Theorem~\ref{thm:window}
and its generalization~\cite{HalpRem}, the 
following compositions are equivalences
\begin{align*}
\cC_{\yY}^{\pm} 
\subset D_{\mathbb{C}^{\ast}}(\yY, 
w)
\stackrel{\rm{res}}{\to}
D_{\mathbb{C}^{\ast}}(\yY_{L_l^{\pm}\rm{\mathchar`-ss}}, 
w).
\end{align*}
On the other hand, 
we have the open immersion
\begin{align*}
(\fF_0)_{L_l^+\rm{\mathchar`-ss}} \times_{\mM(r, d)}\fF^{-1}
\subset \yY_{L_l^{+}\rm{\mathchar`-ss}}
\end{align*}
by the definition of GIT stability.
By Proposition~\ref{prop:MW}
the critical locus of $w$
restricted to $\yY_{L_l^{+}\rm{\mathchar`-ss}}$
is $\zZ_{L_l^{+}\rm{\mathchar`-ss}}$, 
which is contained in 
$(\fF_0)_{L_l^+\rm{\mathchar`-ss}} 
\times_{\mM(r, d)}\fF^{-1}$. 
Therefore the restriction functor
\begin{align*}
\mathrm{res} \colon D_{\mathbb{C}^{\ast}}(\yY_{L_l^{+}\rm{\mathchar`-ss}}, 
w)
\to D_{\mathbb{C}^{\ast}}
((\fF_0)_{L_l^+\rm{\mathchar`-ss}} 
\times_{\mM(r, d)}\fF^{-1}, w)
\end{align*}
is an equivalence (see~\cite[Lemma~5.5]{HLKSAM}). 
Then using the splitting $\gamma_0$ in (\ref{split:G}), 
a version of {K}n\"orrer periodicity
in Theorem~\ref{thm:knoer}
implies the equivalence
\begin{align*}
D^b(M^{\rm{T}}(r, d))=D^b(\zZ_{L_l^{+}\rm{\mathchar`-ss}})
\stackrel{\sim}{\to}
D_{\mathbb{C}^{\ast}}
((\fF_0)_{L_l^+\rm{\mathchar`-ss}} 
\times_{\mM(r, d)}\fF^{-1}, w). 
\end{align*}
Therefore we obtain the first equivalence of (\ref{equiv:CY}). 
The second equivalence of (\ref{equiv:CY}) 
also holds by the same argument, using 
the splitting $\gamma_1$ in (\ref{split:G}) instead of $\gamma_0$. 
\end{proof}

Now Theorem~\ref{thm:Tpair} follows from the following 
theorem together with Proposition~\ref{prop:equiv:CY} 
and the isomorphism (\ref{isom:dual}). 
\begin{thm}\label{thm:pair:window}
For $d>0$, 
we have $\cC_{\yY}^- \subset \cC_{\yY}^{+}$. 
For $d=0$, we have $\cC_{\yY}^-=\cC_{\yY}^{+}$. 
\end{thm}
\begin{proof}
Let us take $p \in \vV \subset M(r, d)$ and consider the diagram (\ref{dia:pair:analytic})
as before. 
Then by the compatibility of linearizations (\ref{line:pullback}), 
the KN stratifications (\ref{KN:Y}) restricted to $\pi_{\yY}^{-1}(\vV)$
coincides with the KN stratifications (\ref{KN:quiver})
for $Q^{\dag}=Q_{E}^{\dag}$ restricted to 
$[h_Q^{-1}(\uU)/G_E]$ under the isomorphism $\xi$. 
Therefore 
for an object $\eE \in D_{\mathbb{C}^{\ast}}(\yY, w)$, 
it is an object in $\cC_{\yY}^{\pm}$ if and only if 
for any point $p \in M(r, d)$ and 
an analytic open neighborhood $p \in \vV \subset M(r, d)$, 
we have 
\begin{align*}
\xi^{\ast}(\eE|_{\pi_{\yY}^{-1}(\vV)}) \in \cC_{\delta, \uU}^{\pm}
\subset D_{\mathbb{C}^{\ast}}([h_{Q}^{-1}(\uU)/G_E], w^{\star})
\end{align*}
where $w^{\star}$ is given by (\ref{wQ}) and 
$\cC_{\delta, \uU}^{\pm}$ are the window subcategories 
(\ref{window:analytic})
for the quiver $Q^{\dag}=Q_{E}^{\dag}$. 
Since $\cC_{\delta, \uU}^- \subset \cC_{\delta, \uU}^+$ 
by Proposition~\ref{prop:window:analytic}, we
conclude $\cC_{\yY}^- \subset \cC_{\yY}^+$
for $d>0$. Similarly we have 
$\cC_{\yY}^-=\cC_{\yY}^+$ for $d=0$. 
\end{proof}

\subsection{Geometry on the stable locus}
In what follows, 
we will describe the explicit form of the fully faithful functor 
constructed in Theorem \ref{thm:Tpair} on the stable locus. 

Let $p \in M(r, d)$ be a point corresponding to 
a stable vector bundle $E$, 
$Q=Q_{E}$ the associated Ext-quiver, 
and $G_{E}=T=\mathbb{C}^*$. 
Note that 
$Q$ has only one vertex with $c$-loops
for $c \cneq \ext^1(E, E)$.  
Let $W$ be a $\mathbb{C}^*$-representation 
appeared in Lemma \ref{lem:decompose}, 
which is of weight one by Lemma~\ref{F:wone}. 
Hence we have 
\[
\mathrm{Rep}(Q^{\dag})_W
=U \times V^+ \times V^-, 
\]
where $U$, 
$V^{+}$, $V^{-}$ are $\mathbb{C}^{\ast}$-representations 
of weights $0, +1, -1$, respectively, given by 
\begin{align*}
U \cneq \Ext^1(E, E), \ V^+ \cneq H^0(E) \times W, \
V^- \cneq H^1(E) \times W^{\vee}.
\end{align*}
Below we set $a:=\rk\mathcal{F}_{0}=\dim V^+$, 
$b:=\rk\mathcal{F}_{1}=\dim V^-$, 
and assume 
that $a \ge b$, which is equivalent to $d \ge 0$. 

We take analytic open neighborhoods
\begin{align*}
p \in \mathcal{V} \subset M(r, d), \ 
0 \in \mathcal{U} \subset \mathrm{Rep}(Q) \sslash G=U
\end{align*}
as in Theorem \ref{thm:analytic}. 
Then we have 
\[
[h_{Q}^{-1}(\mathcal{U})^{\pm}_{\rm{ss}}/G]=Y^{\pm}_{\mathcal{U}}
\subset [h_Q^{-1}(\uU)/G]=\left[(V^+ \times V^-)_{\uU}/\mathbb{C}^{\ast}\right]
\]
where 
$h_{Q}^{-1}(\mathcal{U})^{\pm}_{\rm{ss}} \subset h_Q^{-1}(\uU)$
are semistable locus with respect to 
the characters $\pm \colon \mathbb{C}^{\ast} \to \mathbb{C}^{\ast}$, 
$(-)_{\uU}$ means $(-) \times \uU$ and 
\begin{align*}
&Y^+:=\Tot_{\mathbb{P}(V^+)}(\mathcal{O}_{\mathbb{P}(V^+)}(-1) \otimes V^-), \\
&Y^-:=\Tot_{\mathbb{P}(V^-)}(\mathcal{O}_{\mathbb{P}(V^-)}(-1) \otimes V^+).
\end{align*}
The function $w^{\star}$ on $h_Q^{-1}(\uU)$ given in (\ref{wQ})
is written as 
\begin{align}\label{write:w}
w^{\star}=\sum_{1\le i \le 1, 1\le j \le b}x_i y_j
w_{ij}(\vec{u})
\end{align}
where $\vec{x}$, $\vec{y}$ are coordinates of 
$V^+$, $V^-$, and $w_{ij}(\vec{u}) \in \Gamma(\oO_{\uU})$.

We also set
\begin{align*}
Z:=(V^+\times V^-)\sslash \mathbb{C}^{\ast}=
\Spec \mathbb{C}[x_{i}y_{j} : 1 \leq i \leq a, 1 \leq i \leq b]. 
\end{align*}
Note that $Y^{\pm}$, $Z$
are GIT quotients of $V^+ \times V^-$
by the above $\mathbb{C}^{\ast}$-action, and 
we have the standard toric flip diagram
\begin{align}\label{diagram:flip}
\xymatrix{
Y^{+} \ar[rd] & & \ar[ld] Y^- \\
& Z. &
}
\end{align}


\subsection{Descriptions of the kernel object: local case}
In the case of the previous subsection, 
a pair of the KN stratifications (\ref{KN:quiver}) 
is an elementary wall crossing 
in the sense of \cite[Definition 3.5.1]{MR3895631}, 
and we can describe the corresponding window subcategories 
\[
\mathcal{C}_{\delta, \mathcal{U}}^{\pm} 
\subset 
D_{\mathbb{C}^*}([h_{Q}^{-1}(\mathcal{U})/\mathbb{C}^*], w^{\star})
\]
in the following way: 
For an interval $I \subset \mathbb{R}$, 
define the subcategory 
$\mathcal{W}_{I} \subset 
D_{\mathbb{C}^*}([h_{Q}^{-1}(\mathcal{U})/\mathbb{C}^*], w^{\star})$ 
to be the triangulated subcategory 
generated by the factorizations (\ref{factorization}) 
where $\mathcal{P}_{j}$ are of the form 
\[
\mathcal{P}_{j}
=\bigoplus_{i \in I \cap \mathbb{Z}} 
\mathcal{O}_{h_{Q}^{-1}(\mathcal{U})}(i)^{\oplus l_{i, j}}, 
\quad j=0, 1, 
\quad l_{i, j} \in \mathbb{Z}_{\geq0} 
\]
as $\mathbb{C}^*$-equivariant sheaves. 
Then we have 
$
\mathcal{C}_{\delta, \mathcal{U}}^{\pm}=\mathcal{W}_{I^{\pm}}, 
$
where the intervals $I^{\pm} \subset \mathbb{R}$ 
are defined as follows. 

\begin{align*}
I^-:=\left[-\ceil*{\frac{b}{2}}+1, 
\floor*{\frac{b}{2}}+1 \right),
\quad 
I^+:=\left[-\ceil*{\frac{a}{2}}+1, 
\floor*{\frac{a}{2}}+1 \right).
\end{align*}

As we assumed $a\ge b$, we have 
$\wW_{I^-} \subset \wW_{I^+}$. 
Let us consider the fully faithful functor 
$\Psi \colon 
D_{\mathbb{C}^*}(Y^-_{\mathcal{U}}, w^{\star})
\to D_{\mathbb{C}^*}(Y^+_{\mathcal{U}}, w^{\star})$ 
which fits into the following commutative diagram: 
\[
\xymatrix{
&D_{\mathbb{C}^*}(Y^-_{\mathcal{U}}, w^{\star}) 
\ar@{^{(}->}[r]^{\Psi} 
&D_{\mathbb{C}^*}(Y^+_{\mathcal{U}}, w^{\star}) \\
&\mathcal{W}_{I^-} \ar@{^{(}->}[r] \ar[u]^{\rm{res}}_{\cong} 
&\mathcal{W}_{I^+} \ar[u]_{\rm{res}}^{\cong}. 
}
\]

First we will describe the semi-orthogonal complement 
of the functor $\Psi$. 
Let 
\[
g \colon \mathbb{P}(V^+)_{\mathcal{U}} \to \mathcal{U}, 
\quad 
\iota^+ \colon \mathbb{P}(V^+)_{\mathcal{U}} \hookrightarrow
 Y^+_{\mathcal{U}} 
\]
be the projection, the inclusion as the zero section, 
respectively. 
For an integer $i \in \mathbb{Z}$, 
define the functor $\Upsilon^{i}$ as 
\[
\Upsilon^{i}
:=(- \otimes \mathcal{O}_{Y^+_{\mathcal{U}}}(i)) \circ \iota^+_{*} g^* \colon 
D_{\mathbb{C}^*}(\mathcal{U}, 0) \to 
D_{\mathbb{C}^*}(Y^+_{\mathcal{U}}, w^{\star}). 
\]
Then we have the following result due to 
\cite{MR3327537, MR3895631}. 

\begin{thm}[\cite{MR3327537, MR3895631}]\label{thm:windowsod}
For each integer $i \in \mathbb{Z}$, 
the functor $\Upsilon^{i}$ is fully faithful. 
Furthermore, we have the semi-orthogonal decomposition 
\[
D_{\mathbb{C}^*}(Y^+_{\uU}, w^{\star})
=\left\langle
\Upsilon^{-\lfloor \frac{a}{2} \rfloor}, \cdots, 
\Upsilon^{-\lfloor \frac{b}{2} \rfloor -1}, 
\Psi, 
\Upsilon^{-\lfloor \frac{b}{2} \rfloor}, \cdots, 
\Upsilon^{\lceil \frac{a}{2} \rceil -b-1}
\right\rangle. 
\]
\end{thm}

Next we will determine the kernel of the functor $\Psi$. 
For this purpose, we set $W:=Y^+ \times_{Z}  Y^-$
in the diagram (\ref{diagram:flip}),  
and $q^{\pm} \colon W \to Y^{\pm}$ 
be the projections. 
We have the Fourier-Mukai functor
\begin{align*}
\Phi^{\oO_{W_{\uU}}} =\dR q^{+}_{\ast} \dL q^{-\ast} \colon 
D_{\mathbb{C}^*}(Y^-_{\mathcal{U}}, w^{\star})
\to D_{\mathbb{C}^*}(Y^+_{\mathcal{U}}, w^{\star}).
\end{align*}
Now we prove the following. 

\begin{prop}\label{lem:kernelpsi}
We have an isomorphism of functors
\[
\Psi \cong \left(
- \otimes \mathcal{O}_{Y^+_{\mathcal{U}}}\left(
-\ceil*{\frac{b}{2}} +1 
\right)
\right) \circ 
\Phi^{\mathcal{O}_{W_{\mathcal{U}}}} \circ 
\left(
- \otimes \mathcal{O}_{Y^-_{\mathcal{U}}}
\left(
-\ceil*{\frac{b}{2}}+1 
\right)
\right). 
\]
\end{prop}
\begin{proof}
A version of this result is proved for example in 
\cite[Proposition 5.3]{CIJS15}. 
For the readers' convenience, 
we include the proof here. 
We have the following commutative diagram 
\[
\xymatrix{
&D_{\mathbb{C}^*}(Y^-_{\mathcal{U}}, w^{\star}) 
\ar[rr]^{\otimes \mathcal{O}_{Y^-_{\mathcal{U}}}\left(
-\lceil \frac{b}{2} \rceil +1 \right)} 
\ar@/_24pt/[dd]_{\Psi} 
&
&D_{\mathbb{C}^*}(Y^-_{\mathcal{U}}, w^{\star}) 
\ar@/^24pt/[dd]^{\Psi'} \\
&\mathcal{W}_{I^-} 
\ar[rr]^{\otimes \mathcal{O}\left(
\lceil \frac{b}{2} \rceil -1 
\right)} 
\ar@{^{(}->}[d] \ar[u] 
& 
&\mathcal{W}_{[0, b)} \ar@{^{(}->}[d]  \ar[u] \\
&D_{\mathbb{C}^*}(Y^+_{\mathcal{U}}, w^{\star}) 
\ar[rr]_{\otimes \mathcal{O}_{Y^+_{\mathcal{U}}}\left(
\lceil \frac{b}{2} \rceil -1 \right)}
&
&D_{\mathbb{C}^*}(Y^+_{\mathcal{U}}, w^{\star}). 
} 
\]
We set $\Phi \cneq \Phi^{\mathcal{O}_{W_{\mathcal{U}}}}$
and show an isomorphism 
$\Psi' \cong \Phi$ of functors. 
Take a factorization 
\[
E=\left(
\mathcal{P}_{0} \xrightarrow{f} 
\mathcal{P}_{1} \xrightarrow{g} 
\mathcal{P}_{0}(1) 
\right) \in \mathcal{W}_{[0, b)}, 
\]
with 
\[
\mathcal{P}_{0}=\bigoplus_{i=1}^k \mathcal{O}(m_{i}), 
\quad \mathcal{P}_{1}=\bigoplus_{j=1}^l \mathcal{O}(n_{j}), 
\quad m_{i}, n_{j} \in [0, b), 
\]
and let 
$E^{\pm}:=\eta^{\pm*}_{\mathcal{U}}E 
\in D_{\mathbb{C}^*}(Y^{\pm}_{\mathcal{U}}, w^{\star})$, 
where 
$\eta^{\pm} \colon Y^{\pm} \hookrightarrow 
\left[
(V^+ \times V^-)/\mathbb{C}^*
\right]$ 
are the natural open immersions. 
By definition, we have 
$\Psi'(E^-)=E^+$. 
We need to construct a functorial isomorphism 
\[
\theta_{E} \colon E^+ \xrightarrow{\cong} \Phi(E^-). 
\]
For this, it is enough to prove Lemma~\ref{lem:sublem} below. 
\end{proof}

\begin{lem}\label{lem:sublem}
Take integers $m, n \in [0, b)$ 
and a $\mathbb{C}^*$-equivariant morphism 
$f \colon \mathcal{O}(m) \to \mathcal{O}(n)$
on $(V^+ \times V^-)_{\uU}$. 
Put $f^{\pm}:=\eta^{\pm*}_{\mathcal{U}}(f) \colon 
\mathcal{O}_{Y^{\pm}_{\mathcal{U}}}(\pm m) \to 
\mathcal{O}_{Y^{\pm}_{\mathcal{U}}}(\pm n)$. 
Then there exist isomorphisms 
$\theta_{i} \colon \mathcal{O}_{Y^+_{\mathcal{U}}}(i) \to 
\Phi(\mathcal{O}_{Y^-_{\mathcal{U}}}(-i))$ 
for $i=m, n$, fitting into the commutative diagram 
\begin{equation}\label{eq:functorial}
\xymatrix{
&\Phi(\mathcal{O}_{Y^-_{\mathcal{U}}}(-m)) \ar[r]^{\Phi(f^-)} 
&\Phi(\mathcal{O}_{Y^-_{\mathcal{U}}}(-n)) \\ 
&\mathcal{O}_{Y^+_{\mathcal{U}}}(m) \ar[r]_{f^+} \ar[u]^{\theta_{m}}_{\cong} 
&\mathcal{O}_{Y^+_{\mathcal{U}}}(n). \ar[u]_{\theta_{n}}^{\cong}  
}
\end{equation}
\end{lem}
\begin{proof}
We first recall the descriptions of 
$W$ and the morphisms $q^{\pm} \colon W \to Y^{\pm}$
in terms of GIT quotients. 
We have  
\[
W=\left[
(V^{+*} \times V^{-*} \times \mathbb{C})/(\mathbb{C}^*)^2
\right], 
\]
where $(s_{1}, s_{2}) \in (\mathbb{C}^*)^2$ acts on 
$V^+ \times V^- \times \mathbb{C}$ as 
\[
(s_{1}, s_{2}) \cdot (\vec{x}, \vec{y}, t)
:=(s_{1}\vec{x}, s_{2}^{-1}\vec{y}, s_{1}^{-1}s_{2}t), 
\]
and the morphisms $q^{\pm} \colon W \to Y^{\pm}$ 
are induced by the homomorphisms 
\[
\zeta^{\pm} \colon V^+ \times V^- \times \mathbb{C} 
\to V^+ \times V^-, \quad 
(\mathbb{C}^*)^2 \to \mathbb{C}^* 
\]
defined by 
\begin{align*}
&\zeta^+ \colon (\vec{x}, \vec{y}, t) \mapsto (\vec{x}, t\vec{y}), 
\quad (s_{1}, s_{2}) \mapsto s_{1}, \\
&\zeta^- \colon (\vec{x}, \vec{y}, t) \mapsto (t\vec{x}, \vec{y}), 
\quad (s_{1}, s_{2}) \mapsto s_{2}, 
\end{align*}
respectively.

By the descriptions of the morphisms $\zeta^{\pm}$, 
we have the following commutative diagram: 
\begin{equation}\label{eq:pullz}
\xymatrix{
&\mathcal{O}(0, -m) \ar[r]^{\zeta^{-*}(f)} 
&\mathcal{O}(0, -n) \\
&\mathcal{O}(m, 0) \ar[r]_{\zeta^{+*}(f)} \ar[u]^{t^m} 
&\mathcal{O}(n, 0) \ar[u]_{t^n}. 
}
\end{equation}
Here, $\mathcal{O}(i, j)$ 
denotes the $(\mathbb{C}^*)^2$-equivariant trivial line bundle 
on $\mathcal{U} \times V^+ \times V^- \times \mathbb{C}$ 
of weight $(i, j) \in \mathbb{Z}^2$. 
Pulling back the diagram (\ref{eq:pullz}) to $W_{\mathcal{U}}$, 
we get 
\begin{equation}\label{eq:pullq}
\xymatrix{
&\mathcal{O}_{W_{\mathcal{U}}}(0, -m) \ar[rr]^{q^{-*}(f^-)} 
& &\mathcal{O}_{W_{\mathcal{U}}}(0, -n) \\
&\mathcal{O}_{W_{\mathcal{U}}}(m, 0) \ar[rr]_{q^{+*}(f^+)} \ar[u]^{t^m} 
& &\mathcal{O}_{W_{\mathcal{U}}}(n, 0) \ar[u]_{t^n}. 
}
\end{equation}
In the above diagram, 
the vertical morphisms are induced by the morphism 
\[
\mathcal{O}_{W_{\mathcal{U}}} \xrightarrow{t^d} 
\mathcal{O}_{W_{\mathcal{U}}}(-d, -d)
=\mathcal{O}_{W_{\mathcal{U}}}(dD), 
\]
where $d \in \mathbb{Z}$ is an integer 
and $D \subset W_{\mathcal{U}}$ 
is the exceptional divisor of $q^{\pm}$. 
Furthermore, if $d \in [0, b)$, 
then the morphism 
$\dR q^+_{*}(t^d) \colon \mathcal{O}_{Y^+_{\mathcal{U}}} 
\to \dR p^+_{*}\mathcal{O}_{W_{\mathcal{U}}}(dD)$ 
is an isomorphism. 
Since we assume $m, n \in [0, b)$, 
by applying the functor $\dR q^+_{*}$ 
to the diagram (\ref{eq:pullq}), 
we get the commutative diagram as in (\ref{eq:functorial}). 
\end{proof}

Let us consider the d-critical loci
\begin{align*}
M^{\pm}_{(Q^{\dag}, w), \uU}=\{dw^{\star}=0\}
\subset Y^{\pm}_{\uU}.
\end{align*}
By the diagram (\ref{dia:pair:analytic}) together with 
Proposition~\ref{prop:MW}, 
the above d-critical loci are smooth 
and contained in $\mathbb{P}(Y^{\pm})_{\uU}$. 
Explicitly, using the functions $w_{ij}(\vec{u})$ in (\ref{write:w}), 
we have
\begin{align*}
&M^{+}_{(Q^{\dag}, w), \uU}=\left\{
(\vec{x}, \vec{u}) \in 
\mathbb{P}(V^+)_{\uU} : 
\sum_{i=1}^{a} x_{i}w_{ij}(\vec{u})=0 
\mbox{ for all } 1 \leq j \leq b
\right\}, \\
&M^{-}_{(Q^{\dag}, w), \uU}=\left\{
(\vec{y}, \vec{u}) \in 
\mathbb{P}(V^+)_{\uU} : 
\sum_{j=1}^{b} x_{i}w_{ij}(\vec{u})=0 
\mbox{ for all } 1 \leq i \leq a
\right\}.
\end{align*}
Let us consider the following diagram 
\[
\xymatrix{
&A^{\pm} \ar@{^{(}->}[r]^{i^{\pm}} \ar[d]_{p^{\pm}} \ar@{}[dr]|\square
&Y^{\pm}_{\mathcal{U}} \ar[d] \ar[r]^{w^{\star}} & \mathbb{C} \\ 
&M^{\pm}_{(Q^{\dagger}, w), \mathcal{U}} \ar@{^{(}->}[r] 
&\mathbb{P}(V^{\pm}_{\mathcal{U}}). & 
}
\]
By Theorem~\ref{thm:knoer}, we have the equivalences
\begin{align*}
i_{\ast}^{\pm} p^{\pm \ast} \colon 
D^b(M^{\pm}_{(Q^{\dag}, w), \mathcal{U}}) 
\stackrel{\sim}{\to}
D_{\mathbb{C}^{\ast}}(Y_{\uU}^{\pm}, w^{\star}).
\end{align*}
Now let us consider the functor
$\Theta_M$ defined by the following 
commutative diagram. 
\[
\xymatrix{
&D^b(M^-_{(Q^{\dag}, w), \mathcal{U}}) 
\ar[r]^{\Theta_{M}} \ar[d]_{i^-_{*}p^{-*}}^{\cong} 
&D^b(M^+_{(Q^{\dag}, w), \mathcal{U}}) \ar[d]^{i^+_{*}p^{+*}}_{\cong} \\
&D_{\mathbb{C}^*}(Y^-_{\mathcal{U}}, w^{\star}) 
\ar[r]_{\Psi}
&D_{\mathbb{C}^*}(Y^+_{\mathcal{U}}, w^{\star}). 
}
\]

To determine the kernel of the functor $\Theta_{M}$, 
we recall the results from \cite{Todsemi}. 
Let $B \subset 
\left( 
\mathbb{P}(V^+) \times \mathbb{P}(V^-)
\right)_{\mathcal{U}}$ be the subvariety defined by 
\[
B:=\left\{
(\vec{x}, \vec{y}, \vec{u}) \in 
\left( 
\mathbb{P}(V^+) \times \mathbb{P}(V^-)
\right)_{\mathcal{U}} : 
\sum_{i, j} x_{i}y_{j}w_{ij}(\vec{u})=0
\right\}. 
\]
We have the following diagram: 
\[
\xymatrix{
&E  \ar@{^{(}->}[r]^{\tilde{i}} \ar[d]_{\tilde{p}}\ar@{}[dr]|\square
 & W_{\mathcal{U}} \ar[d]^{\rm{pr}} \ar[r]^{\widetilde{w}} &\mathbb{C}\\ 
&B \ar@{^{(}->}[r] 
&\left( 
\mathbb{P}(V^+) \times \mathbb{P}(V^-)
\right)_{\mathcal{U}}. &
}
\]
Here $\widetilde{w}$ is the pull-back of 
$w^{\star}$ by 
$q^{\pm} \colon W_{\uU} \to Y_{\uU}^{\pm}$. 
We also define subvarieties 
$F^{\pm} \subset \left( 
\mathbb{P}(V^+) \times \mathbb{P}(V^-)
\right)_{\mathcal{U}}$ 
as follows. 
\begin{align*}
&F^+:=\left\{
(\vec{x}, \vec{y}, \vec{u}) \in 
\left( 
\mathbb{P}(V^+) \times \mathbb{P}(V^-)
\right)_{\mathcal{U}} : 
\sum_{i=1}^{a} x_{i}w_{ij}(\vec{u})=0 
\mbox{ for all } 1 \leq j \leq b
\right\}, \\
&F^-:=\left\{
(\vec{x}, \vec{y}, \vec{u}) \in 
\left( 
\mathbb{P}(V^+) \times \mathbb{P}(V^-)
\right)_{\mathcal{U}} : 
\sum_{j=1}^{b} y_{j}w_{ij}(\vec{u})=0 
\mbox{ for all } 1 \leq i \leq a
\right\}. 
\end{align*}
We obtain the following diagram: 
\[
\xymatrix{
&F^{\pm} \ar@{^{(}->}[r]^{k^{\pm}} \ar[d]_{r^{\pm}} &B \\
&M^{\pm}_{(Q^{\dag}, w), \mathcal{U}}. 
}
\]
Here $r^{\pm}$ are the projections and $k^{\pm}$ are
the closed immersions. 
Define functors 
$\Theta^{\pm}$ as 
\[
\Theta^{\pm}:=k^{\pm}_{*} \circ r^{\pm*} \colon 
D^b(M^{\pm}_{(Q^{\dag}, w), \mathcal{U}}) \to D^b(B), 
\]
and let 
$\Theta^{\pm}_{\dR}:=\dR r^{\pm}_{*} \circ (k^{\pm})^{!}$ 
be the right adjoints of the functors $\Theta^{\pm}$. 
We have the following results: 

\begin{lem}[{\cite[Lemma~3.4, Lemma~3.5]{Todsemi}}]\label{lem:commpull}
The following diagram is commutative: 
\begin{align}\label{eq:commpull}
\xymatrix{
D^b(M^{\pm}_{(Q^{\dag}, w), \mathcal{U}}) 
\ar[r]^-{\Theta^{\pm}} \ar[d]_-{i^{\pm}_{*}p^{\pm*}} 
&D^b(B) \ar[d]^-{\tilde{i}_{*}\tilde{p}^*} \ar[r]^-{\Theta^{\pm}_{\dR}} & 
D^b(M^{\pm}_{(Q^{\dag}, w), \mathcal{U}}) \ar[d]^-{i^{\pm}_{*}p^{\pm*}} \\
D_{\mathbb{C}^*}(Y^{\pm}, w^{\star}) 
\ar[r]_-{\dL q^{\pm*}} 
&D_{\mathbb{C}^*}(W_{\mathcal{U}}, \tilde{w}) \ar[r]_-{\dR q^{\pm}_{\ast}} &
D_{\mathbb{C}^*}(Y^{\pm}, w^{\star}). 
}, 
\end{align}
\end{lem}


\begin{prop}[{\cite[Proposition 3.6]{Todsemi}}]\label{prop:tkernel}
We have an isomorphism of functors
\begin{align}\label{eq:tkernel}
\Theta^+_{\dR} \circ \Theta^- 
\cong \otimes \mathcal{O}_{M^+_{(Q^{\dag}, w), \mathcal{U}}}(b-1) \circ 
\Phi^{\mathcal{P}} \circ 
\otimes \mathcal{O}_{M^-_{(Q^{\dag}, w), \mathcal{U}}}(-1)[1-b], 
\end{align}
where we put $
\mathcal{P}:=
\mathcal{O}_{M^-_{(Q^{\dag}, w), \mathcal{U}} \times_{\mathcal{U}} 
M^+_{(Q^{\dag}, w), \mathcal{U}}}$, 
and $\Phi^{\pP}$ is the Fourier-Mukai functor with kernel $\pP$. 
\end{prop}

Let $\pi^{\pm} \colon M^{\pm}_{(Q^{\dagger}, w), \mathcal{U}} \to \mathcal{U}$ 
be the projections. 
For every integer $i \in \mathbb{Z}$, 
we define the functor 
$\overline{\Upsilon}^i$ as 
\[
\overline{\Upsilon}^i:=
\otimes \mathcal{O}_{M^+_{(Q^{\dagger}, w), \mathcal{U}}}(i) \circ \dL\pi^{+*} 
\colon D^b(\mathcal{U}) \to D^b(M^+_{(Q^{\dagger}, w), \mathcal{U}}). 
\]

\begin{lem}[{\cite[Lemma 3.7]{Todsemi}}]\label{lem:commu}
For each integer $i \in \mathbb{Z}$, 
the following diagram is commutative 
\begin{equation}\label{eq:commu}
\xymatrix{
&D^b(\mathcal{U}) \ar[rr]^-{\overline{\Upsilon}^{i+b}[-b]} \ar@{=}[d] 
& &D^b(M^+_{(Q^{\dagger}, w), \mathcal{U}}) \ar[d]^-{i^{+}_{*}p^{+*}}  \\
&D^b(\mathcal{U}) \ar[rr]_-{\Upsilon^{i}} 
& &D_{\mathbb{C}^*}(Y^+_{\mathcal{U}}, w^{\star}). 
}
\end{equation}
\end{lem}

Now we can describe the functor $\Theta_{M}$ explicitly. 

\begin{prop}\label{prop:localkernel}
We have an isomorphism of functors 
\begin{equation}\label{eq:localkernel}
\Theta_{M}\cong \otimes \mathcal{O}_{M^+_{(Q^{\dag}, w), \mathcal{U}}}
\left(
\floor*{\frac{b}{2}}
\right) \circ 
\Phi^{\mathcal{P}} \circ 
\otimes \mathcal{O}_{M^-_{(Q^{\dag}, w), \mathcal{U}}}
\left( 
-\ceil*{\frac{b}{2}} 
\right)[1-b]. 
\end{equation}
\end{prop}
\begin{proof}
By Proposition~\ref{lem:kernelpsi}
and the commutativity of the diagrams 
(\ref{eq:commpull}),
we have 
\begin{align*}
&\quad \Psi \circ i^-_{*}p^{-*}\\
&=\otimes \mathcal{O}_{Y^+_{\mathcal{U}}}\left(
-\ceil*{\frac{b}{2}} +1
\right) \circ 
\Phi^{\mathcal{O}_{W_{\mathcal{U}}}} \circ 
\otimes \mathcal{O}_{Y^-_{\mathcal{U}}}
\left(
-\ceil*{\frac{b}{2}} +1
\right) \circ i^-_{*}p^{-*} \\
&=\otimes \mathcal{O}_{Y^+_{\mathcal{U}}}\left(
-\ceil*{\frac{b}{2}} +1 
\right) 
\circ \dR q^+_* \circ \dL q^{-*} \circ 
\left(\mathcal{O}_{Y^-_{\mathcal{U}}}\left(
-\ceil*{\frac{b}{2}} +1 
\right) \otimes 
i^-_*p^{-*}(-) 
\right) \\
&\cong \otimes \mathcal{O}_{Y^+_{\mathcal{U}}}\left(
-\ceil*{\frac{b}{2}} +1 
\right) \circ 
\dR q^+_*\left(
\mathcal{O}_{W_{\mathcal{U}}}\left(
0, -\ceil*{\frac{b}{2}} +1
\right) \otimes 
\dL q^{-*} \circ i^-_*p^{-*}(-) 
\right) \\
&\cong \otimes \mathcal{O}_{Y^+_{\mathcal{U}}}\left(
-\ceil*{\frac{b}{2}} +1 
\right) \circ \dR q^+_*\left(
\mathcal{O}_{W_{\mathcal{U}}}\left(
0, -\ceil*{\frac{b}{2}} +1
\right) \otimes 
\tilde{i}_*\tilde{p}^* \circ \Theta^-(-)
\right) \\
&\cong \otimes \mathcal{O}_{Y^+_{\mathcal{U}}}\left(
-\ceil*{\frac{b}{2}} +1 
\right) \circ \dR q^+_*\tilde{i}_*\tilde{p}^*\left(
\mathcal{O}_{B}\left(
0, -\ceil*{\frac{b}{2}} +1
\right) \otimes 
\Theta^-(-)
\right) \\ 
&\cong \otimes \mathcal{O}_{Y^+_{\mathcal{U}}}\left(
-\ceil*{\frac{b}{2}} +1 
\right) \circ 
i^+_*p^{+*} \circ \Theta^+_{\dR} \circ \Theta^- 
\circ \otimes \mathcal{O}_{M^-_{(Q^{\dag}, w), \mathcal{U}}}
\left(-\ceil*{\frac{b}{2}} +1 \right) \\
&\cong i^+_*p^{+*} \circ \left( 
\otimes \mathcal{O}_{M^-_{(Q^{\dag}, w), \mathcal{U}}}
\left(-\ceil*{\frac{b}{2}} +1 \right) \circ 
\Theta^+_{\dR} \circ \Theta^- 
\circ \otimes \mathcal{O}_{M^-_{(Q^{\dag}, w), \mathcal{U}}}
\left(-\ceil*{\frac{b}{2}} +1 \right)
\right). 
\end{align*}
Hence the assertion follows from the equation 
(\ref{eq:tkernel}). 
\end{proof}

\subsection{Descriptions of the kernel object: global case}
Denote by 
\begin{align*}
\Phi_{M} \colon D^b(M^{\rm{DT}}(r, d)) \hookrightarrow D^b(M^T(r, d))
\end{align*}
the fully faithful functor given in Theorem \ref{thm:Tpair}. 
We consider the base change of the diagram (\ref{diagram:MT})
to the stable locus 
$M^{\rm{st}}(r, d) \subset M(r, d)$: 
\[
\xymatrix{
&M^{\rm{T}, \rm{st}}(r, d) \ar[rd]_{\pi^+} & &M^{\rm{DT, st}}(r, d) \ar[ld]^{\pi^-} \\
& &M^{\rm{st}}(r, d) &. 
}
\]
By the construction of $\Phi_M$, 
it restricts to the fully-faithful functor
\begin{align*}
\Phi_{M^{\rm{st}}} \colon D^b(M^{\rm{DT}, \rm{st}}(r, d)) 
\hookrightarrow D^b(M^{T, \rm{st}}(r, d)). 
\end{align*}
Let 
$W^{\rm{st}}$ be the fiber product
\begin{align*}
W^{\rm{st}}:=
M^{\rm{T}, \rm{st}}(r, d) \times_{M^{\rm{st}}(r, d)} 
M^{\rm{DT}, \rm{st}}(r, d).
\end{align*}
Suppose that there exists a universal bundle on 
$C \times M^{\rm{st}}(r, d)$, or equivalently
\begin{align*}
\iota_M^{-1}(M^{\rm{st}}(r, d))= M^{\rm{st}}(r, d)
\times B\mathbb{C}^{\ast}
\end{align*}
where $\iota_M$ is the map (\ref{M:good}). 
For an integer $i \in \mathbb{Z}$, 
let 
$\oO(i)$ be the line bundle on $B\mathbb{C}^{\ast}$
given by a one dimensional $\mathbb{C}^{\ast}$-representation
with weight $i$, and denote by 
$\mathcal{O}_{M^{\rm{T, st}}(r, d)}(i)$
its pull-back to $M^{\rm{T, st}}(r, d)$
by the map in the diagram (\ref{diagram:MT}). 
We define 
$\Upsilon^i_{M^{\rm{st}}}$ to be the functor 
\[
\Upsilon^i_{M^{\rm{st}}}:=
\otimes \mathcal{O}_{M^{\rm{T, st}}(r, d)}(i) \circ \dL\pi^{+*} 
\colon D^b(M^{\rm{st}}(r, d)) \to D^b(M^{\rm{T, st}}(r, d)). 
\]
As a summary of the discussions in this subsection, 
we have the following result. 
\begin{thm}\label{globalkernel}
The following statements hold. 
\begin{enumerate}
\item There exists a line bundle $L$ on $W^{\rm{st}}$ 
and an isomorphism of functors 
\[
\Phi_{M^{\rm{st}}}\cong
\Phi^{i_{W*}L}[1-b], 
\]
where 
$i_{W} \colon W^{\rm{st}} \hookrightarrow 
M^{\rm{T}, \rm{st}}(r, d) \times M^{\rm{DT}, \rm{st}}(r, d)$ 
is the inclusion. 
\item For each integer $i \in \mathbb{Z}$, 
the functor 
\[
\Upsilon^i_{M^{\rm{st}}} \colon 
D^b(M^{\rm{st}}(r, d)) \to D^b(M^{\rm{T, st}}(r, d))
\]
is fully faithful. 

\item We have the semi-orthogonal decomposition 
\[
D^b(M^{\rm{T, st}}(r, d))=\left\langle
\Upsilon^{-\lfloor \frac{a}{2} \rfloor+b}_{M^{\rm{st}}}, \cdots, 
\Upsilon^{\ceil*{\frac{b}{2}}-1}_{M^{\rm{st}}}, 
\Phi_{M^{\rm{st}}}, 
\Upsilon^{\ceil*{\frac{b}{2}}}_{M^{\rm{st}}}, \cdots, 
\Upsilon^{\lceil \frac{a}{2} \rceil-1}_{M^{\rm{st}}}
\right\rangle. 
\]
\end{enumerate}
\end{thm}

Before the proof, we need a preparation. 
Let $Y$ be a smooth quasi-projective variety, 
$X_{i}$ be smooth varieties with projective morphisms 
$X_{i} \to Y$ ($i=1, 2, 3$). 
For objects $\mathcal{E} \in D^b(X_{1} \times X_{2})$, 
$\mathcal{F} \in D^b(X_{2} \times X_{3})$, 
supported on the fiber products 
$X_{1} \times_{Y} X_{2}$, $X_{2} \times_{Y} X_{3}$, 
we define 
\[\mathcal{F} \circ \mathcal{E}:=
\dR p_{13*}(p_{12}^*\mathcal{E} \otimes^{\dL} p_{23}^*\mathcal{F}) 
\in D^b(X_{1} \times X_{3}). 
\]
Here, $p_{ij} \colon X_{1} \times X_{2} \times X_{3} \to X_{i} \times X_{j}$ 
denotes the projections ($i, j=1, 2, 3$). 
Then we have an isomorphism of functors 
\[
\Phi^{\mathcal{F}} \circ \Phi^{\mathcal{E}} 
\cong \Phi^{\mathcal{F} \circ \mathcal{E}} 
\colon D^b(X_{1}) \to D^b(X_{3}). 
\]

\begin{lem}\label{lem:uniquekernel}
Let $Y$ be a smooth quasi-projective variety, 
$X_{1}, X_{2}$ smooth varieties, projective over $Y$. 
Let $\mathcal{E}, \mathcal{F} \in D^b(X_{1} \times X_{2})$ 
be objects supported on the fiber product 
$X_{1} \times_{Y} X_{2}$. 
Assume that we have an isomorphism of functors 
\[
\Phi^{\mathcal{E}} \cong \Phi^{\mathcal{F}} 
\colon D^b(X_{1}) \hookrightarrow D^b(X_{2})
\] 
and that they are fully faithful functors. 
Then we have an isomorphism 
$\mathcal{E} \cong \mathcal{F}$. 
\end{lem}
\begin{proof}
A similar statement was proved in \cite[Lemma 4.1]{Todwidth} 
when the functors are equivalences. 
In particular, by {\it loc. cit.}, we have 
an isomorphism 
\[
\mathcal{F}^* \circ \mathcal{E} \cong \mathcal{O}_{\Delta_{X_{1}}}. 
\]
Here, the object $\mathcal{F}^* \in D^b(X_{1} \times X_{2})$ 
is defined as 
\[
\mathcal{F}^*:=\dR \mathcal{H}om_{X_{1} \times X_{2}}\left(
\mathcal{F}, \mathcal{O}_{X_{1} \times X_{2}} 
\right) \otimes p_{1}^*\omega_{X_{1}}[\dim X_{1}], 
\]
where $p_{1} \colon X_{1} \times X_{2} \to X_{1}$ is the projection. 

On the other hand, there exists an object 
$\mathcal{G} \in D^b(X_{2} \times X_{2})$ and an exact triangle 
\begin{equation}\label{eq:adjoint}
\mathcal{F} \circ \mathcal{F}^* \to \mathcal{O}_{\Delta_{X_{2}}} \to 
\mathcal{G}
\end{equation}
which induces the adjoint counit 
$\Phi^{\mathcal{F}} \circ \Phi^{\mathcal{F}^*} \to \rm{Id}_{D^b(X_{2})}$ 
(cf. \cite[Theorem 3.1]{AL12}). 
By the construction, the functor 
$\Phi^{\mathcal{G}} \colon D^b(X_{2}) \to D^b(X_{2})$ 
coincides with the projection functor 
$D^b(X_{2})\to (\Phi^{\mathcal{F}}(D^b(X_{1})))^{\perp}$ 
to the right orthogonal complement. 
In particular, we have 
\[
\Phi^{\mathcal{G} \circ \mathcal{E}} 
\cong \Phi^{\mathcal{U}} \circ \Phi^{\mathcal{E}}=0 
\colon D^b(X_{1}) \to (\Phi^{\mathcal{F}}(D^b(X_{1})))^{\perp} 
\]
and hence $\mathcal{G} \circ \mathcal{E} \cong 0$. 
Applying $(-) \circ \mathcal{E}$ 
to the exact triangle (\ref{eq:adjoint}), 
we get the exact triangle 
\[
\mathcal{F} \cong \mathcal{F} \circ \mathcal{F}^* \circ \mathcal{E}
\to \mathcal{E} \to \mathcal{G} \circ \mathcal{E} \cong 0
\]
in $D^b(X_{1} \times X_{2})$, i.e., 
$\mathcal{F} \cong \mathcal{E}$. 
\end{proof}

\begin{proof}[Proof of Theorem \ref{globalkernel}]
The second and third statements hold 
by its analytic local version Theorem \ref{thm:windowsod}
(see the argument of~\cite[Theorem~4.5]{Todsemi}). 
We prove the first statement. 
To simplify the notation,
we set
\begin{align*}
M:=M^{\rm{st}}(r, d), \
M^+ := M^{\rm{T, st}}(r, d), \ 
M^-:=M^{\rm{DT, st}}(r, d), \  
W:=W^{\rm{st}}.
\end{align*}
Denote by $\rho \colon W \to M$ the projection. 
Let $\mathcal{E} \in D^b(M^- \times M^+)$ 
be the kernel object of the functor $\Phi_{M^{\rm{st}}}[b-1]$. 
By Proposition \ref{prop:localkernel}, there exist 
an analytic open covering 
$\{\mathcal{U}_{i}\}_{i \in \mathbb{Z}}$ of $M$ 
and line bundles $L_{i}$ on $\rho^{-1}(\mathcal{U}_{i})$ 
satisfying the following property: 
For each $i \in \mathbb{Z}$, we have the commutative diagram 
\[
\xymatrix{
&D^b\left(\left(\pi^-\right)^{-1}\left(\mathcal{U}_{i}\right)\right) 
\ar@{^{(}->}[r]^{\Phi^{L_{i}}} 
&D^b\left(\left(\pi^+\right)^{-1}\left(\mathcal{U}_{i}\right)\right) \\
&D^b(M^-) \ar@{^{(}->}[r]^{\Phi^{\mathcal{E}}} \ar[u]^{\rm{res}} 
&D^b(M^+) \ar[u]_{\rm{res}}. 
}
\]
From this, we can see that the object 
$\mathcal{E}$ is supported on $W$, 
and we have an isomorphism of functors 
\[
\Phi^{\mathcal{E}_{i}} \cong \Phi^{L_{i}} \colon 
D^b\left(\left(\pi^-\right)^{-1}\left(\mathcal{U}_{i}\right)\right) \to 
D^b\left(\left(\pi^-\right)^{-1}\left(\mathcal{U}_{i}\right)\right) 
\]
for each $i \in \mathbb{Z}$, 
where we put 
$\mathcal{E}_{i}:=\mathcal{E}|_{\rho^{-1}(\mathcal{U}_{i})}$. 
Hence by Lemma \ref{lem:uniquekernel}, 
we conclude that the object $\mathcal{E}$ 
is isomorphic to some line bundle $L$ on $W$. 
\end{proof}

\section{Comparison with ADHM sheaves}\label{sec:ADHM}
In this section, we will explain about 
the hidden Calabi-Yau three structure 
behind the d-critical structure 
on the moduli spaces of Thaddeus 
pairs constructed in Theorem \ref{thm:dflip}.

\subsection{ADHM sheaves}
In this subsection, we recall the definition 
and the basic facts about ADHM sheaves 
studied in \cite{Dia, Dia1, CDP}. 
Let $C$ be a smooth projective curve 
of genus $g$, and fix line bundles 
$M_{1}, M_{2}$ on $C$ such that 
$M_{1} \otimes M_{2} \cong \omega_{C}^{\vee}$. 
First we define the notion of Higgs bundles. 

\begin{defi}\label{def:H}
\begin{enumerate}
\item A {\it Higgs bundle} is a triplet 
$(E, \Phi_{1}, \Phi_{2})$, 
where $E$ is a locally free sheaf on $C$, 
and 
\[
\Phi_{i} \colon E \otimes M_{i} \to E, 
\]
are morphisms of coherent sheaves 
satisfying the relation 
\[
\Phi_{1} \circ (\Phi_{2} \otimes \id_{M_{1}}) 
-\Phi_{2} \circ (\Phi_{1} \otimes \id_{M_{2}}) 
=0. 
\]
\item Let $(E, \Phi_{1}, \Phi_{2})$ be a Higgs bundle. 
A subsheaf $F \subset E$ is called 
{\it $\Phi$-invariant} if for $i=1, 2$, 
we have 
$\Phi_{i}(F \otimes M_{i}) \subset F$. 
\item A Higgs bundle $(E, \Phi_{1}, \Phi_{2})$ is called 
{\it semistable} if for every $\Phi$-invariant subsheaf $F \subset E$, 
we have $\mu(F) \leq \mu(E)$. 
\end{enumerate}
\end{defi}

We denote by $H(r, d)$ the good 
moduli space of 
semistable Higgs bundles $(E, \Phi_{1}, \Phi_{2})$ 
satisfying the condition 
\[
(\rk(E), \chi(E))=(r, d). 
\]

Similarly as in Definition \ref{def:T}, 
we define the following notions: 

\begin{defi}\label{def:adhm}
\begin{enumerate}
\item A semistable ADHM sheaf is 
a quadruplet 
\begin{align*}
(E, \Phi_{1}, \Phi_{2}, \psi)
\end{align*}
 consisting of 
a semistable Higgs bundle $(E, \Phi_{1}, \Phi_{2})$ 
and a morphism $\psi \colon \mathcal{O}_{C} \to E$ 
such that there is no non-zero subbundle $F \subsetneq E$ 
with $\mu(F)=\mu(E)$ and $\psi(\mathcal{O}_{C}) \subset F$. 
\item A semistable dual ADHM sheaf is 
a quadruplet $(E', \Phi'_{1}, \Phi'_{2}, \psi')$ consisting of 
a semistable Higgs bundle $(E', \Phi'_{1}, \Phi'_{2})$ 
and a morphism $\psi' \colon E \to \omega_{C}$ 
such that there is no non-zero subbundle $F' \subsetneq E'$ 
with $\mu(F)=\mu(E)$ and $\psi'(F')=0$. 
\end{enumerate}
\end{defi}

\begin{rmk}\label{rmk:adhm}
In \cite[Definition 2.1]{Dia1}, the notion of 
{\it $\delta$-semistable ADHM sheaves} are introduced 
for a real number $\delta \in \mathbb{R}$. 
Our notion of semistable (resp. dual) ADHM sheaves 
is equivalent to the notion of $\delta$-semistable ADHM sheaves 
for $0 < \delta \ll 1$ (resp. $-1 \ll \delta <0$). 
\end{rmk}

We denote by $M^{\rm{ADHM}}(r, d)$ 
(resp. $M^{\rm{DADHM}}(r, d)$) 
the moduli space of semistable (resp. dual) ADHM sheaves 
$(E, \Phi_{1}, \Phi_{2}, \psi)$ with 
$[E, \Phi_{1}, \Phi_{2}] \in H(r, d)$. 

We have the following diagram 
of quasi-projective schemes over $\mathbb{C}$: 
\begin{equation}\label{diagram:adhm}
\xymatrix{
&M^{\rm{ADHM}}(r, d) \ar_{\sigma^{+}}[rd] & 
&M^{\rm{DADHM}}(r, d) \ar^{\sigma^{-}}[ld] \\
& &H(r,d). & 
}
\end{equation}

\begin{rmk}\label{rmk:wccy}
Let $X:=\Tot_C(M_{1}^{\vee} \oplus M_{2}^{\vee})$ 
be a non-compact Calabi-Yau threefold. 
It is well-known that giving a Higgs bundle 
is equivalent to giving a compactly supported 
one dimensional pure sheaf on $X$. 
Hence we can think the moduli spaces 
$M^{\rm{ADHM}}(r, d), M^{\rm{DADHM}}(r, d)$, and $H(r, d)$ 
as the moduli spaces of objects in $D^b(X)$. 
In fact, the diagram (\ref{diagram:adhm}) can be described as 
the wall crossing diagram of the moduli spaces of 
objects in $D^b(X)$, which was studied in \cite{Tolim2}. 
\end{rmk}

We have the following result. 

\begin{thm}[{\cite[Theorem 9.13]{Toddbir}}]
For $d>0$, 
the diagram (\ref{diagram:adhm}) is a d-critical flip. 
For $d=0$, it is a d-critical flop. 
\end{thm}

\subsection{Torus action on ADHM sheaves}
In this subsection, we consider the natural $\mathbb{C}^*$-action 
on the moduli space of semistable (dual) ADHM sheaves. 
We will see that the moduli space of (dual) Thaddeus pairs appear 
as one of the connected components of $\mathbb{C}^*$-fixed locus. 

\begin{defi}
We define a $\mathbb{C}^*$-action 
on $M^{\rm{ADHM}}(r, d)$ as follows: 
for $t \in \mathbb{C}^*$ and 
$[E, \Phi_{1}, \Phi_{2}, \psi] \in M^{\rm{ADHM}}(r, d)$, 
\[
t \cdot (E, \Phi_{1}, \Phi_{2}, \psi)
:=(E, t\Phi_{1}, t^{-1}\Phi_{2}, \psi). 
\]
We define a $\mathbb{C}^*$-action on 
$M^{\rm{DADHM}}(r, d)$ similarly. 
\end{defi}

We have the following theorem due to \cite{Dia1}: 

\begin{thm}[{\cite[Theorem 1.5]{Dia1}}]\label{thm:torusfix}
The $\mathbb{C}^*$-fixed loci 
$M^{\rm{ADHM}}(r, d)^{\mathbb{C^*}}$ and 
$M^{\rm{DADHM}}(r, d)^{\mathbb{C^*}}$ 
are projective schemes over $\mathbb{C}$. 
\end{thm}

Now we can prove the following: 
\begin{prop}\label{prop:torusfix}
The moduli spaces 
$M^{T}(r, d), M^{\rm{DT}}(r, d)$ 
are one of the connected components 
of the torus fixed loci 
$M^{\rm{ADHM}}(r, d)^{\mathbb{C^*}}$, 
$M^{\rm{DADHM}}(r, d)^{\mathbb{C^*}}$, 
respectively. 
\end{prop}
\begin{proof}
We only prove the assertion for $M^T(r, d)$. 
By the definitions of stability and the $\mathbb{C}^*$-action, 
we have the inclusion 
$\phi \colon M^T(r, d) \hookrightarrow M^{\rm{ADHM}}(r, d)^{\mathbb{C}^*}$. 
Furthermore, both of these schemes are projective 
by Theorem \ref{thm:torusfix}. 
Hence it is enough to show that the inclusion $\phi$ 
is an open immersion. 

Take an element 
\[
p=I^{\bullet}=(\mathcal{O}_{C} \xrightarrow{s} E) \in M^T(r, d). 
\]
We think $I^{\bullet}$ as an object 
in the derived category $D^b(C)$ so that 
$\mathcal{O}_{C}$ is located in degree $0$. 
As mentioned in Remark \ref{rmk:wccy}, 
the moduli space $M^{\rm{ADHM}}(r, d)$ 
can be described as the moduli space of objects in $D^b(X)$. 
Then we have 
$J^{\bullet}:=\phi(p)
=(\mathcal{O}_{X} \xrightarrow{t} \iota_{*}E) \in D^b(X)$, 
where $\iota \colon C \to X$ is the zero section 
and the map $t$ is given by the composition 
\[
t \colon \mathcal{O}_{X} \twoheadrightarrow \iota_{*}\mathcal{O}_{C} 
\xrightarrow{\iota_{*}(s)} \iota_{*}E. 
\]
We will show that the tangent map 
\[
T_{p}\phi \colon T_{p}M^T(r, d) \to 
T_{\phi(p)}M^{\rm{ADHM}}(r, d)^{\mathbb{C}^*} 
\]
is an isomorphism. 
By the deformation-obstruction theory for pairs, 
the tangent spaces are given as follows: 
\begin{align*}
&T_{p}M^T(r, d)=\Hom_{C}(I^{\bullet}, E), \\
&T_{\phi(p)}M^{\rm{ADHM}}(r, d)^{\mathbb{C}^*}
=\Hom_{X}(J^{\bullet}, \iota_{*}E)^{\mathbb{C}^*}
\cong \Hom_{C}(\dL \iota^*J^{\bullet}, E)^{\mathbb{C}^*}. 
\end{align*}
Note that by the stability condition, 
the section $s \in H^0(C, E)$ is non-zero, and 
we have the short exact sequence 
\[
0 \to \mathcal{O}_{C} \to E \to Q \to 0 
\]
for some coherent sheaf $Q \in \Coh(C)$. 
Hence we have 
\[
\Hom_{C}(I^{\bullet}, E) \cong \Ext^1_{C}(Q, E). 
\]

To compute the vector space 
$\Hom_{C}(\dL \iota^*J^{\bullet}, E)^{\mathbb{C}^*}$, 
let us apply the functor $\dL \iota^*$ 
to the exact triangle 
\[
\iota_{*}E[-1] \to J^{\bullet} \to \mathcal{O}_{X}. 
\]
We get the long exact sequence 
\[
\xymatrix{
&0 \ar[r] &E \otimes (M_{1} \oplus M_{2}) \ar[r] 
&\mathcal{H}^0(\dL \iota^*J^{\bullet}) \ar[r] &\mathcal{O}_{C} \\
&\ar[r]^{\iota^*(t)=s} &E \ar[r] &\mathcal{H}^1(\dL \iota^*J^{\bullet}) \ar[r] &0. 
}
\]
Hence we have 
\[
\mathcal{H}^0(\dL \iota^*J^{\bullet}) \cong E \otimes (M_{1} \oplus M_{2}), 
\quad \mathcal{H}^1(\dL \iota^*J^{\bullet}) \cong Q, 
\]
and $\mathcal{H}^i(\dL \iota^*J^{\bullet})=0$ 
for $i \geq 2$. 
Furthermore, by the definition of the $\mathbb{C}^*$-action, 
we have the vanishing 
$\Hom_{C}(E \otimes (M_{1} \oplus M_{2}), E)^{\mathbb{C}^*}=0$. 
We conclude that 
\[
\Hom_{C}(\dL \iota^*J^{\bullet}, E)^{\mathbb{C}^*} \cong \Ext^1_{C}(Q, E) 
\]
as required. 
\end{proof}

Recall that our $\mathbb{C}^*$-action 
on ADHM sheaves are defined so that 
its weights are $1, -1$ on $M_{1}, M_{2}$, respectively. 
In particular, the action preserves the isomorphism 
$M_{1} \otimes M_{2} \cong \omega_{C}^{\vee}$. 
Hence by taking fixed loci, 
the relative d-critical charts on the diagram (\ref{diagram:adhm}) induce 
the relative d-critical charts on the wall crossing diagram of 
moduli spaces of Thaddeus pairs considered in Theorem \ref{thm:dflip}.

\appendix
\section{Review on d-critical birational geometry}
Here we recall the basic notions in d-critical birational geometry 
introduced by the second author \cite{Toddbir, Todsemi}. 

\subsection{D-critical loci}
In this subsection, we quickly review about 
the notion of (analytic) d-critical loci introduced by Joyce. 
For more detail, see his original paper \cite{JoyceD}. 
Let $M$ be a complex analytic space. 
Then there exists a sheaf $\mathcal{S}_{M}$ 
of $\mathbb{C}$-vector spaces on $M$ 
satisfying the following property: 
for any analytic open subset $U \subset M$ 
and a closed immersion $i \colon U \hookrightarrow Y$ 
into a smooth complex manifold $Y$, 
there exists an exact sequence 
\[
0 \to \mathcal{S}_{M}|_{U} \to \mathcal{O}_{Y}/I^2 \to \Omega_{Y}/I^2, 
\]
where $I \subset \mathcal{O}_{Y}$ is 
the ideal sheaf defining $U \subset Y$. 
Furthermore, there exists a direct sum decomposition 
\[
\mathcal{S}_{M}=\mathbb{C}_{M} \oplus \mathcal{S}_{M}^0. 
\]

The following is the basic example. 
\begin{exam}\label{exam:locald}
Let $Y$ be a complex manifold, 
$U \subset Y$ a closed analytic subspace. 
Assume that there exists a holomorphic function 
$w \colon Y \to \mathbb{C}$ satisfying 
\begin{equation}\label{eq:crit}
U=\{dw=0\} \subset Y, \quad 
w|_{U^{\rm{red}}}=0. 
\end{equation}
Then the section 
\begin{equation}\label{eq:sec}
s:=w+I^2 \in H^0(U, \mathcal{O}_{Y}/I^2) 
\end{equation}
is in fact an element of $H^0(U, \mathcal{S}^0_{U})$. 
Here, $I=(dw) \subset \mathcal{O}_{Y}$ 
is the ideal sheaf defining $U \subset Y$. 
\end{exam}

\begin{defi}
A {\it d-critical locus} is a pair $(M, s)$ 
consisting of a complex analytic space $M$ 
and a section $s \in H^0(M, \mathcal{S}^0_{M})$ 
satisfying the following property: 
for every point $x \in M$, 
there exist an open neighborhood $x \in U \subset M$ of $x$, 
a closed embedding $i \colon U \hookrightarrow Y$ 
into a complex manifold $Y$, 
and a holomorphic function $w \colon Y \to \mathbb{C}$ 
satisfying the conditions (\ref{eq:crit}) such that 
the restriction $s|_{U}$ can be written as (\ref{eq:sec}) 
in Example \ref{exam:locald}. 
We call the data 
$(U, Y, w, i)$ as a {\it d-critical chart}, 
and the section $s$ as a {\it d-critical structure}. 
\end{defi}

The following definition is a relative version of d-critical charts. 
\begin{defi}\label{def:A3}
Let $(M, s)$ be a d-critical locus, 
$\pi \colon M \to A$ a morphism of analytic spaces. 
Let $U \subset A$ be an open subset. 
Assume that there exists a following commutative diagram 
\[
\xymatrix{
&\pi^{-1}(U) \ar@{^{(}->}[r]^{i} \ar[d]_{\pi} &Y \ar[d]^{f} \ar[rd]^{w} &\\
&U \ar@{^{(}->}[r]_{j} &Z \ar[r]_{g} &\mathbb{C}, 
}
\]
where $f$ is a morphism of analytic spaces with $Y$ smooth, 
$i, j$ are closed embeddings, and $g$ is a holomorphic function. 
If the data $(\pi^{-1}(U), Y, w, i)$ defines a 
d-critical chart of $(M, s)$, 
then we call it as a {\it $\pi$-relative d-critical chart}. 
\end{defi}

\subsection{D-critical birational transforms}
We define the d-critical analogue of birational transforms. 
For the standard terminologies in birational geometry, 
we refer to~\cite{KMM, KM}. See also \cite[Section 2]{Toddbir}. 

\begin{defi}\label{defi:A4}
Let $(M^{\pm}, s^{\pm})$ be d-critical loci, 
$\pi^{\pm} \colon M^{\pm} \to A$ be morphisms of analytic spaces. 
Then the diagram 
\begin{equation}\label{eq:ddiagram}
\xymatrix{
&M^+ \ar[rd]_{\pi^+} & &M^- \ar[ld]^{\pi^-} \\
& &A &
}
\end{equation}
is called a 
{\it d-critical flip, d-critical flop}, at a point $p \in A$ 
if there exist an open neighborhood 
$p \in U \subset A$ and $\pi^{\pm}$-relative d-critical charts 
\[
\xymatrix{
&(\pi^{\pm})^{-1}(U) \ar@{^{(}->}[r]^-{i^{\pm}} \ar[d]_-{\pi^{\pm}} 
&Y^{\pm} \ar[d]^-{f^{\pm}} \ar[rd]^-{w^{\pm}} &\\
&U \ar@{^{(}->}[r]_-{j} &Z \ar[r]_-{g} &\mathbb{C}, 
}
\]
with $g$ and $j$ are independent on $\pm$,  
such that the diagram 
\[
Y^+ \xrightarrow{f^+} Z \xleftarrow{f^-} Y^- 
\]
is a (usual) flip, flop, 
respectively. 

We call the digram (\ref{eq:ddiagram}) as a 
{\it d-critical d-critical flip, d-critical flop}, 
if it satisfies the corresponding condition 
at any point $p \in A$. 
\end{defi}

We can see that d-critical birational transforms defined above 
decrease {\it virtual canonical line bundles} on d-critical loci. 
See \cite[Section 3]{Toddbir} for more detail.

\bibliographystyle{amsalpha}
\bibliography{math}

\end{document}